\documentclass[letterpaper, 10 pt, conference]{ieeeconf}  

\IEEEoverridecommandlockouts                              
\usepackage[OT1]{fontenc} 
\usepackage[english]{babel} 
\usepackage[driverfallback=dvipdfmx,hidelinks]{hyperref}
\usepackage{amsmath,amssymb,mathrsfs}
\usepackage{cleveref}
\usepackage{bm}
\usepackage{array}
\usepackage{multirow}
\usepackage[hang]{footmisc}

\usepackage{booktabs}
\usepackage{algorithm}
\usepackage{algpseudocodex}
\usepackage{comment}
\usepackage{subfigure}
\usepackage{tabularx}
\usepackage{bibentry}
\usepackage[noadjust]{cite}
\usepackage{bbm}
\usepackage{menukeys}
\usepackage[short]{optidef}

\usepackage{amsthm}
\newtheorem{problem}{Problem}
\newtheorem{theorem}{Theorem}
\newtheorem{proposition}{Proposition}

\newtheorem{lemma}{Lemma}
\newtheorem{remark}{Remark}

\newcommand{\norm}[1]{\lVert#1\rVert}

\newcommand{\R}{\mathbb{R}}

\newcommand{\Z}{\mathbb{Z}}
\newcommand{\E}{\mathbb{E}}
\newcommand{\diff}{\mathrm{d}}

\newcommand{\cov}[1]{\mathrm{Cov}[#1 ]}
\newcommand{\normfs}[1]{\lVert#1\rVert} 
\let\P\relax
\newcommand{\P}[1]{\mathbb{P}\left[#1\right]}

\newcommand{\ols}[1]{\mskip.5\thinmuskip\overline{\mskip-.5\thinmuskip {#1} \mskip-.5\thinmuskip}\mskip.5\thinmuskip} 

\crefname{align}{Eq.}{Eqs.}
\crefname{equation}{Eq.}{Eqs.}
\crefname{figure}{Fig.}{Figs.}
\crefname{table}{Table}{Tables}
\crefname{theorem}{Theorem}{Theorems}
\crefname{definition}{Definition}{Definitions}
\crefname{lemma}{Lemma}{Lemmas}
\crefname{remark}{Remark}{Remarsks}
\crefname{assumption}{Assumption}{Assumptions}
\crefname{proof}{Proof}{Proofs}
\crefname{algorithm}{Algorithm}{Algorithms}
\crefname{problem}{Problem}{Problems}
\crefname{proposition}{Proposition}{Propositions}
\crefname{corollary}{Corollary}{Corollaries}
\crefname{section}{Section}{Sections}

\usepackage[numbered,framed]{matlab-prettifier}
\usepackage{filecontents}

\graphicspath{ {./img/} }

\def\shortOrFull{2} 

\def\doi{To Be Assigned}

\title{
Chance-Constrained Control for Safe Spacecraft Autonomy: \\ Convex Programming Approach
}

\author{Kenshiro Oguri
\thanks{This work was supported by Purdue University and the U.S. Air Force Office of Scientific Research through FA9550-23-1-0512.
K.~Oguri is with the School of Aeronautics and Astronautics, Purdue University, IN 47907, USA ({\tt koguri@purdue.edu}).
}}

\if\shortOrFull2 
\usepackage{tikz}
\usepackage{textcomp}
\newcommand\copyrighttext{%
  \scriptsize \textcopyright 2024 IEEE. Personal use of this material is permitted.
  Permission from IEEE must be obtained for all other uses, in any current or future
  media, including reprinting/republishing this material for advertising or promotional
  purposes, creating new collective works, for resale or redistribution to servers or
  lists, or reuse of any copyrighted component of this work in other works.
  DOI: \doi.
  Accepted for 2024 IEEE American Control Conference (ACC).
  }
\newcommand\copyrightnotice{%
\begin{tikzpicture}[remember picture,overlay]
\node[anchor=south,yshift=10pt] at (current page.south) {\fbox{\parbox{\dimexpr\textwidth-\fboxsep-\fboxrule\relax}{\copyrighttext}}};
\end{tikzpicture}%
}
\fi

\begin{document}

\maketitle
\thispagestyle{empty}
\pagestyle{empty}

\if\shortOrFull2 
\copyrightnotice
\fi

\begin{abstract}
This paper presents a robust path-planning framework for safe spacecraft autonomy under uncertainty and develops a computationally tractable formulation based on convex programming.
We utilize chance-constrained control to formulate the problem.
It provides a mathematical framework to solve for a sequence of control policies that minimizes a probabilistic cost under probabilistic constraints with a user-defined confidence level (e.g., safety with 99.9\% confidence).
The framework enables the planner to directly control state distributions under operational uncertainties while ensuring the vehicle safety.
This paper
rigorously formulates the safe autonomy problem,
gathers and extends techniques in literature to accommodate key cost/constraint functions that often arise in spacecraft path planning,
and
develops a tractable solution method.
The presented framework is demonstrated via
two representative numerical examples:
safe autonomous rendezvous
and
orbit maintenance in cislunar space,
both under uncertainties due to navigation error from Kalman filter, execution error via Gates model, and imperfect force models.
\end{abstract}



\section{Introduction}
\label{sec:intro}

Safety is crucial in spacecraft autonomy.
As any space vehicles must operate under various operational uncertainties, safety assurance under such uncertainties is prerequisite for any autonomous guidance navigation control (GNC) algorithms to be deployed on real-world space vehicles. 
Notable uncertainties in spacecraft GNC include
navigation (estimation) error,
maneuver execution error,
and
imperfect force models.
It is a challenging task to ensure the safety of autonomous vehicles under such uncertainties with severe constraints in space operations, such as
the limited onboard computation,
communication bandwidth,
and
stringent safety constraints (e.g., keep-out zone, approach cone, tube/box constraints about nominal trajectories)
\cite{Starek2016a}.

To address the challenge, this paper leverages recent advances in chance-constrained control
\cite{Ono2013,Okamoto2018a,Ridderhof2020}.
Chance-constrained control is a class of stochastic control that seeks a sequence of control policies that minimizes a probabilistic cost while imposing probabilistic constraints with a user-defined confidence level.
A chance constraint is defined as:
\begin{align} \begin{aligned}
\P{\mathrm{safe}} \geq 1 - \varepsilon,\quad
\label{eq:SOCconstraints}
\end{aligned} \end{align}
where 
$\varepsilon$ is a risk bound (e.g., $\varepsilon=0.01$ for $99\%$ confidence).
Chance-constrained approaches directly control and impose safety constraints on state distributions under uncertainty, 
which is in sharp contrast to conventional GNC algorithms,
such as those based on Gauss equations 
\cite{Gaias2015b}, 
convex programming
\cite{Acikmese2007,Lu2013}, 
predictor-corrector
\cite{Ito2020,Lu2008}, 
sliding-mode control
\cite{Furfaro2013a}, 
and 
model predictive control (MPC)
\cite{Weiss2015a,DiCairano2012}.
A key advantage over other robust control approaches, such as robust MPC
\cite{Kothare1996,Kuwata2007,Buckner2018,Oestreich2023}, 
lies in its ability to handle \textit{un}bounded distributions, which are common in spacecraft GNC, e.g., Gaussian-distributed state estimate from a Kalman filter.

\cref{f:framework} illustrates the safe autonomy framework envisioned in this study.
In this framework, a chunk of chance-constrained control policies, $\pi_{1:N}$ ($N$: planning horizon), are computed on ground via convex programming, infrequently uploaded to spacecraft, and run on-board to calculate maneuvers every time a new state estimate becomes available via filtering.
Maneuvers derived from the policies are safe by design under uncertainties that are modeled in the planner.

\begin{figure}[tb]
\centering \includegraphics[width=0.9\linewidth]{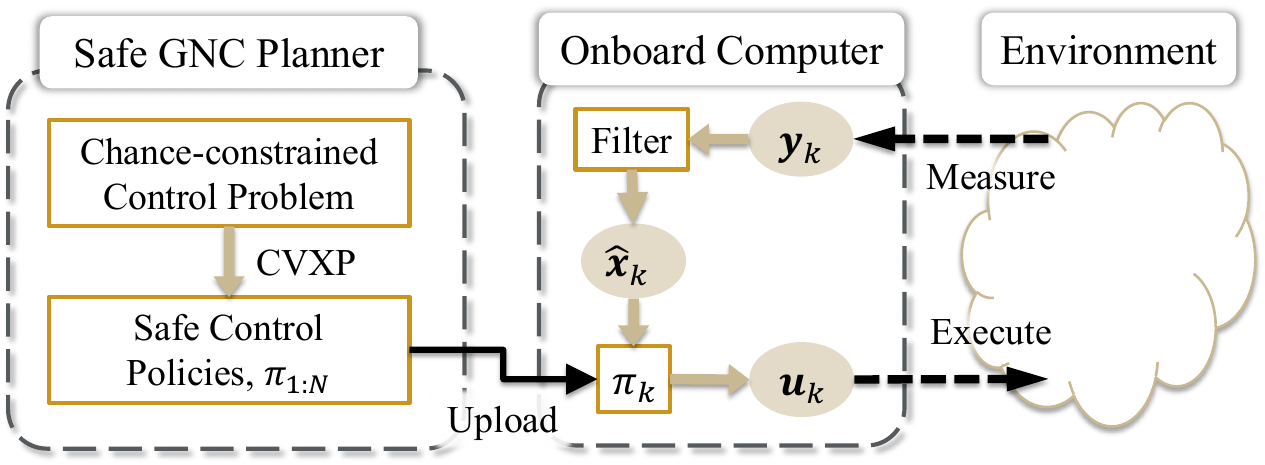}
\vspace{-5pt}
\caption{\label{f:framework} Safe autonomy framework based on chance-constrained control}
\vspace{-15pt}
\end{figure}

The main contributions of this paper are threefold.
First,
this paper rigorously formulates the safe spacecraft autonomy concept envisioned in \cref{f:framework} as an output-feedback chance-constrained control problem.
Second, we extend techniques in literature to incorporate notable cost function and constraints in spacecraft GNC and to exploit the Markovian property of the system.
Third, this paper demonstrates the autonomy framework via
two representative scenarios:
safe autonomous rendezvous
and
station-keeping on a cislunar near rectilinear halo orbit (NRHO).

\subsubsection*{Notation}
$\Z_{0}^{N}$ denotes the set of integers from $0$ to $N$.
$\norm{\cdot}_p $ denotes $l_p$ norm for a vector.
For a matrix $A$, $\norm{A}_2 $ is its spectral norm.
For a matrix $A \succeq 0$, $A^{1/2}$ is lower-triangular and satisfies $A = A^{1/2} [A^{1/2}]^\top $, and $\lambda_{\max}(A) (=\norm{A^{1/2}}_2^2) $ is the largest eigenvalue of $A$.
$\P{\cdot}$, $\E{[\cdot]}$, and $\cov{\cdot}$ are probability, expectation, and covariance operators.


\section{Orbital Mechanics under Control}
\label{sec:EoM}

\subsection{Generic Formulation}
Let $\bm{x}\in\R^{n_x} $ be the spacecraft orbital state and $\bm{u}\in\R^{n_u} $ be the control input.
In general, spacecraft orbital dynamics under control are expressed as a control-affine system:
\begin{align} \begin{aligned}
\dot{\bm{x}} =& \bm{f}(\bm{x}, \bm{u}, t) = 
\bm{f}_0(\bm{x},t)
+
B\bm{u} ,\quad 
B = 
\begin{bmatrix}
0_{n_u\times n_u} \\ I_{n_u}
\end{bmatrix}
\label{eq:genericEoM}
\end{aligned} \end{align}
where $\bm{f}_0(\cdot) $ represents orbital dynamics under no control.
The control input $\bm{u}$ may model impulsive maneuvers (delta-V) or continuous acceleration.
We characterize $\bm{u}(t)$ by a finite number of control inputs $\bm{u_k}, k\in\Z_{0}^{N-1}$ as:
\begin{align} \begin{aligned}
\bm{u}(t) = 
\begin{cases}
\sum_{k=0}^{N-1} \bm{u}_k \delta(t - t_k) & \text{(impulsive)} \\
\bm{u}_k,\quad \forall t\in[t_k,t_{k+1}) & \text{(continuous)}
\end{cases}
\label{eq:control}
\end{aligned} \end{align}
where $\delta(\cdot)$ is the Dirac delta function, and the zero-order-hold (ZOH) is assumed for continuous control.

\if\shortOrFull1 
\input{2-1-short}
\fi

\if\shortOrFull2 
\subsection{Specific Formulations}
Appropriate equations of motion (EoMs) to use depend on the orbital regime and operation scenario.
Let us review some popular forms of $\bm{f}_0(\cdot)$ with underlying assumptions.

\subsubsection{Perturbed 2BP}
The perturbed two-body problem (2BP) expressed in Cartesian coordinates is a common dynamical system in spacecraft GNC.
The state is defined as $\bm{x} = [\bm{r}^\top, \bm{v}^\top]^\top $, where $\bm{r},\bm{v}$ are the spacecraft position and velocity,
and $\bm{f}_0(\cdot)$ is given by \cite{Battin1999}:
\begin{align} \begin{aligned}
\bm{f}_0(\bm{x},t) = 
\begin{bmatrix}
\bm{v} \\ -\frac{\mu}{\norm{\bm{r}}_2^3}\bm{r}
\end{bmatrix}
+ B\bm{a}(\bm{x}, t)
\label{eq:keplerianCartesian}
\end{aligned} \end{align}
where
$\mu$ is the gravitational parameter of the body, and $\bm{a}$ is perturbing acceleration due to the surrounding environment.

\subsubsection{CWH equation}
Clohessy-Wiltshire Hill (CWH) equation approximates \cref{eq:keplerianCartesian} to model orbital motion in proximity of another object (called \textit{chief} satellite) .
CWH equation is derived with assumptions that
(i) the chief and our spacecraft is sufficiently close, and
(ii) the chief orbit is circular about the gravitational body.
The state is defined as $\bm{x} = [\bm{r}^\top, \bm{v}^\top]^\top $, where $\bm{r},\bm{v}\in\R^3$ are the spacecraft position and velocity in a rotating frame that rotates with the chief orbit about the gravitational body.
$\bm{f}_0(\cdot)$ is given by \cite{Alfriend2010}:
\begin{align} \begin{aligned}
&\bm{f}_0(\bm{x},t) = A\bm{x} + B\bm{a}(\bm{x},t)
,\quad
A = 
\begin{bmatrix}
0_{3\times 3} & I_3 \\
A_{21} & A_{22}
\end{bmatrix}
,\\ 
&A_{21} = 
\begin{bmatrix}
3 n^2 & 0 & 0 \\
0 & 0 & 0 \\
0 & 0 & -n^2
\end{bmatrix}
,\quad 
A_{22} = 
\begin{bmatrix}
0 & 2n & 0 \\
-2n & 0 & 0 \\
0 & 0 &  0
\end{bmatrix}
\label{eq:CWH}
\end{aligned} \end{align}
where
$n = \sqrt{{\mu}/{r_0^3}}$, and $r_0$ is the chief orbit radius.

\subsubsection{CR3BP}
Circular Restricted Three-Body Problem (CR3BP) is the simplest possible expression of the three-body problem.
It assumes that
two massive bodies with gravitational parameters $\mu_1$ and $\mu_2$ are in a circular orbit about their barycenter.
The state vector is defined as $\bm{x} = [\bm{r}^\top, \bm{v}^\top]^\top $, where $\bm{r} = [x, y, z]^\top$ and $\bm{v} = [\dot{x}, \dot{y}, \dot{z}]^\top $ are \textit{non-dimensional} spacecraft position and velocity in the rotating frame that rotates with the two massive bodies.
$\bm{r}$ and $\bm{v}$ can be dimensionalized by multiplying the characteristic length $l_*$ and velocity $v_*$, where $l_*$ is the distance between the two massive bodies, $v_* = l_* / t_*$, and $t_*$ (characteristic time) is given by $t_* = \sqrt{l_*^3 / (\mu_1 + \mu_2)} $.
$\bm{f}_0(\cdot)$ is given by \cite{Szebehely1967a}:
\begin{align} 
&\bm{f}_0(\bm{x},t) = 
\begin{bmatrix}
\bm{v} \\
\bm{f}_a(\bm{x})
\end{bmatrix}
+ 
B\bm{a}
\label{eq:CR3BP}
\end{align}
$\bm{f}_a(\cdot)$ is given in \cref{eq:CR3BP2}, where
$\mu = \mu_2 / (\mu_1 + \mu_2)$,
$r_1 = \sqrt{(x+\mu)^2 + y^2 + z^2} $, and $r_2 = \sqrt{(x-1+\mu)^2 + y^2 + z^2} $.
\begin{align} 
\bm{f}_a(\bm{x}) =
\begin{bmatrix}
2 \dot{y}+ x - \frac{(1-\mu)(x+\mu)}{r_1^3}-\frac{\mu(x-1+\mu)}{r_2^3} \\
- 2 \dot{x} + y - \frac{(1-\mu) y}{r_1^3}-\frac{\mu y}{r_2^3}\\
\frac{-(1-\mu) z}{r_1^3}-\frac{\mu z}{r_2^3}
\end{bmatrix}
\label{eq:CR3BP2}
\end{align}

\fi

\section{Chance-Constrained Control Problem}
\label{sec:formulation}

\subsection{Uncertainty Modeling}
\label{sec:uncertaintyModel}

\subsubsection{Initial state dispersion}
We model the distributions of $\hat{\bm{x}} $ (estimate of $\bm{x}$) and $\widetilde{\bm{x}}(= \bm{x} - \hat{\bm{x}} )$ at $t_0$ via independent Gaussian distributions as
\begin{align}
\hat{\bm{x}}_0^{-} \sim\mathcal{N}(\ols{\bm{x}}_0, \hat{P}_{0^-}),\  
\widetilde{\bm{x}}_0 \sim\mathcal{N}(0, \widetilde{P}_{0^-}),
\end{align}
which implies $\bm{x}_0 = \hat{\bm{x}}_0^{-} + \widetilde{\bm{x}}_0 \sim\mathcal{N}(\ols{\bm{x}}_0, \hat{P}_{0^-} + \widetilde{P}_0)$.

\subsubsection{Maneuver execution error}
The Gates model \cite{Gates1963} is a common approach to execution error modeling.
For a given $\bm{u}_k$, it models the control error, $\widetilde{\bm{u}}_k$, as \cite{Oguri2021e,Kumagai2023a}:
\begin{align}\begin{aligned}
&
\widetilde{\bm{u}}(t) = 
\begin{cases}
\sum_{k=0}^{N-1} \widetilde{\bm{u}}_k \delta(t - t_k) & \text{(impulsive)} \\
\widetilde{\bm{u}}_k,\forall t\in[t_k,t_{k+1}) & \text{(continuous)}
\end{cases},\\
&
\widetilde{\bm{u}}_k = G_{\mathrm{exe}}(\bm{u}_k) \bm{w}_{\mathrm{exe},k},\quad  
\bm{w}_{\mathrm{exe},k} \sim \mathcal{N}(0, I_3),
\\
&
G_{\mathrm{exe}}(\bm{u}_k) = T(\bm{u}_k) P_\mathrm{gates}^{1/2}(\bm{u}_k)
\label{eq:executionError}
\end{aligned}\end{align}
where
$\bm{w}_{\mathrm{exe},k} $ are i.i.d.~standard Gaussian vectors, and
\begin{align}\begin{aligned}
&
P_\mathrm{gates}(\cdot) =
\mathrm{diag}(\sigma_p^2, \sigma_p^2,  \sigma_m^2 )
,\ 
T(\cdot) = 
\begin{bmatrix}
\hat{\bm{S}} & \hat{\bm{E}} & \hat{\bm{Z}}
\end{bmatrix}
,\\
&
\hat{\bm{Z}} = \frac{\bm{u}_k}{\norm{\bm{u}_k}_2} 
,\ 
\hat{\bm{E}} = \frac{[0,0,1]^\top \times \hat{\bm{Z}}}{\norm{[0,0,1]^\top \times \hat{\bm{Z}}}_2}
,\ 
\hat{\bm{S}} = \hat{\bm{E}} \times \hat{\bm{Z}}
\label{eq:GatesModel}
\end{aligned}\end{align}
Here, 
$\sigma_p^2 = \sigma_3^2 + \sigma_4^2 \norm{\bm{u}_k}_2^2 $ and $\sigma_m^2 = \sigma_1^2 + \sigma_2^2 \norm{\bm{u}_k}_2^2$.
Note that $\{\sigma_1, \sigma_2\} $ are \{fixed, proportional\} magnitude errors while $\{\sigma_3, \sigma_4\}$ are \{fixed, proportional\} pointing errors.

\subsubsection{Stochastic acceleration}
Stochastic acceleration due to imperfect force modeling can be naturally modeled by a Brownian motion: $G(\bm{x})\diff\bm{w}(t)$, where $G(\cdot) $ is the intensity of disturbances;
$\diff\bm{w}(t) $ is a standard Brownian motion vector, i.e., $\E{[\diff\bm{w}]}=0_3 $ and $\E{[\diff\bm{w}(t)\diff\bm{w}^\top(t)]} = I_{3} \diff t  $.

\subsubsection{Navigation uncertainty}
In essence, navigation is a filtering process with discrete observations, modeled as:
\begin{align}\begin{aligned}
\bm{y}_k =& \bm{f}_{\mathrm{obs}}(\bm{x}_k) + G_{\mathrm{obs}}(\bm{x}_k) \bm{w}_{\mathrm{obs},k},
\label{eq:SOCobservations}
\end{aligned}\end{align}
where
$\bm{f}_{\mathrm{obs}}(\cdot) $ and $G_{\mathrm{obs}}(\cdot) $ are the measurement function and noise intensity, respectively;
$\bm{w}_{\mathrm{obs},k}\in\R^{n_y} $ is an i.i.d. standard Gaussian vector.
A navigation solution at time $t_k$, denoted by $\hat{\bm{x}}_k$, is expressed in terms of its probability density function (pdf) conditioned on all the past measurements $\bm{y}_{1:k}$:
\begin{align}\begin{aligned}
&\hat{\bm{x}}_k \sim f_\mathrm{pdf}(\bm{x}_k | \bm{y}_{1:k})=
\mathcal{F}(\bm{x}_k, \bm{y}_{1:k})
\label{eq:SOCfiltering}
\end{aligned}\end{align}
where $\mathcal{F}(\cdot)$ denotes the filtering process.
Filtering typically utilizes the innovation process $\widetilde{\bm{y}}_k^{-}$, defined as:
\begin{align}\begin{aligned}
\widetilde{\bm{y}}_k^{-} = \bm{y}_k - \bm{f}_{\mathrm{obs}}(\hat{\bm{x}}_k^-)
\label{eq:innovationProcess}.
\end{aligned}\end{align}

\subsubsection{Nonlinear Stochastic System}
The evolution of stochastic state is naturally modeled by nonlinear stochastic differential equation (SDE)
as:
\begin{align}\begin{aligned}
\diff\bm{x} =& 
[\bm{f}(\bm{x}, \bm{u}, t) + B\widetilde{\bm{u}} ] \diff t + G(\bm{x})\diff\bm{w}(t)
\label{eq:SDE}
\end{aligned}\end{align}

\subsection{Orbit Control under Uncertainty}

\subsubsection{Cost Function}

In stochastic settings, the classical form $J = \int_{t_0}^{t_f} L(\bm{u}) \diff t $ ($L$: Lagrangian cost) is not well defined because $\bm{u}$, hence $L(\cdot)$, may be now stochastic.
Instead, we minimize the integral of the $p$-quantile of $L(\cdot)$, i.e.,
\begin{align} \begin{aligned}
J = \int_{t_0}^{t_f} Q_{L(\bm{u})}(p) \diff t
\label{eq:SOCobjective99fuel}
\end{aligned} \end{align}
where 
$Q_{X}(p)$ is the quantile function of a random variable $X$ evaluated at probability $p$, formally defined as:
\begin{align}
Q_{X}(p) = \min\{a\in\R \mid \P{X \leq a} \geq p\}.
\label{eq:defQuantile}
\end{align}
This paper is focused on 
minimizing $\Delta V_{99}$, i.e., $99\%$ quantile of fuel cost, corresponding to $L(\cdot) = \norm{\bm{u}}_2$ with $p=0.99$.

\subsubsection{Path chance constraints}
As our state and control are subject to uncertainty, constraints are not deterministic anymore and need to be treated probabilistically.
We use chance constraints \cref{eq:SOCconstraints} to replace classical path constraints.

This paper is focused on chance constraints imposed at discrete epochs, although there are studies that extend the concept to continuous-time chance constraints \cite{Oguri2019b}.

A simple yet versatile form of state chance constraints is an intersection of hyperplane constraints:
\begin{align} \begin{aligned}
\P{\bm{a}_j^\top \bm{x}_k + b_j \leq 0 ,\ \forall j} \geq 1 - \varepsilon_{x},\  k\in\Z_{0}^{N}
\label{eq:stoStateConst}
\end{aligned} \end{align}
which can conservatively represent any convex feasible regions, including box constraints about a reference trajectory.

If a tube-like constraint about a reference trajectory is preferred over box constraints, we could also consider
\begin{align} \begin{aligned}
\P{\norm{H(\bm{x}_k - \bm{x}_k^*)}_2 \leq d_{\max} } \geq 1 - \varepsilon_x,\ k\in\Z_{0}^{N}
\label{eq:stoStateConstTube}
\end{aligned} \end{align}
where 
$H\in\R^{n_h\times n_x} $ extracts specific elements from $\bm{x}$ (e.g., $H = [I_3, 0_{3\times3}] $ to extract position),
$\bm{x}_k^* $ is the reference state at the epoch, and $d_{\max}$ is the state deviation bound.

Let us then consider control chance constraints.
A common control constraint is the magnitude constraint, $\norm{\bm{u}_k}_2 \leq u_{\max} $, whose chance-constraint counterpart is given by:
\begin{align} \begin{aligned}
\P{\norm{\bm{u}_k}_2 \leq u_{\max} } \geq 1 - \varepsilon_u,\ k\in\Z_{0}^{N-1}
\label{eq:stoControlConst}
\end{aligned} \end{align}
In addition, constraints on control change rate may be crucial when the time for an attitude maneuver is a bottleneck to thrust direction change;
such a constraint is given by:
\begin{align} \begin{aligned}
\P{\norm{\bm{u}_{k+1} - \bm{u}_k}_2 \leq \Delta u_{\max} } \geq 1 - \varepsilon_u,\ k\in\Z_{0}^{N-2}
\label{eq:stoControlRateConst}
\end{aligned} \end{align}

\subsubsection{Distributional terminal constraints}
A natural extension of deterministic terminal constraints would ensure the spacecraft to nominally arrive at the target $\bm{x}_f$ within some prescribed accuracy represented by the final covariance $P_f$:
\begin{align} \begin{aligned}
\E[{\bm{x}}_N] - \bm{x}_f = 0,\quad P_N \preceq P_f
\label{eq:SOCterminal}
\end{aligned} \end{align}
where
$P_N$ denotes the covariance of the $N$-th state.

\subsubsection{State-triggered chance constraints}
State-triggered constraints may arise in many spacecraft GNC problems, especially when we have multiple phases in a mission.
In deterministic form, a state-triggered constraint (STC) is expressed as:
if $g_{\mathrm{stc}}(\bm{x}, \bm{u}) < 0$, then $c_{\mathrm{stc}}(\bm{x}, \bm{u}) \leq 0 $.
This STC is shown in \cite{Szmuk2020} to be logically equivalent to 
\begin{align}
h_{\mathrm{stc}}(\bm{x}, \bm{u}) \triangleq - \min\{g_{\mathrm{stc}}(\bm{x}, \bm{u}), 0 \}\cdot c_{\mathrm{stc}}(\bm{x}, \bm{u}) \leq 0
\label{eq:stateTriggered}
\end{align}
The equivalence can be understood by noting that \cref{eq:stateTriggered} implies that, if $g_{\mathrm{stc}}(\cdot) < 0 $, then $h_{\mathrm{stc}}(\cdot) = - g_{\mathrm{stc}}(\cdot) c_{\mathrm{stc}}(\cdot) \leq 0 $, and hence $c_{\mathrm{stc}}(\cdot) \leq 0$;
on the other hand, if $g_{\mathrm{stc}}(\cdot) \geq 0 $, then $h_{\mathrm{stc}}(\cdot) = 0$ regardless of the value of $c_{\mathrm{stc}}(\cdot)$.

An extension of this concept to a chance constraint with risk level $\epsilon$ can be expressed as:
\begin{align}
\text{if } 
\E[g(\bm{x}, \bm{u})] < 0
,\text{ then }
\P{c(\bm{x}, \bm{u}) \leq 0} \geq 1 - \epsilon
\label{eq:stateTriggeredCC}
\end{align}
which means that $\P{c(\bm{x}, \bm{u}) \leq 0} \geq 1 - \epsilon$ is imposed when the trigger condition $g(\bm{x}, \bm{u}) < 0$ is satisfied in expectation.

In this paper, we consider an approach-cone constraint that is triggered when our satellite is near the origin (e.g., chief satellite).
An approach cone can be represented by a second-order cone $\norm{A_{\mathrm{cone}} \bm{r}}_2 \leq \bm{b}_{\mathrm{cone}}^\top \bm{r}$;
for instance, if the satellite is allowed to approach the chief from ${+}y$ direction, $\sqrt{r_x^2 + r_z^2} \leq r_y \tan{\theta_{\max}} $, where $\theta_{\max}$ is half of the cone angle.
Thus,
we can specialize \cref{eq:stateTriggeredCC} as follows:
\begin{align}\begin{aligned}
&\text{if} 
&&\E\left[\norm{H_r\bm{x}_k}_2 \right]  < r_{\mathrm{trigger}}
,
\\
&
\text{then}
&&\P{\norm{A_{\mathrm{cone}} H_r \bm{x}_k}_2 \leq \bm{b}_{\mathrm{cone}}^\top H_r \bm{x}_k } \geq 1 - \epsilon_{x}
\label{eq:stateTriggeredCCApproachCone}
\end{aligned}\end{align}
where
$H_r\in\R^{3 \times n_x} $ extracts the position vector from $\bm{x}$, and
$r_{\mathrm{trigger}}$ is the critical radius about the origin.

\subsubsection{Control policy}
Since spacecraft never has access to the perfect state knowledge, control policies must calculate maneuvers using imperfect navigation solutions from \cref{eq:SOCfiltering}.
We model the control policy generically as:
\begin{align}\begin{aligned}
\bm{u}_k = \pi_k(\hat{\bm{x}}_k; \Omega_k),
\label{eq:SOCPolicy}
\end{aligned}\end{align}
where
$\Omega_k$ is a set of parameters that parameterize $\pi_k$.

\subsection{Original Chance-Constrained Orbit Control Problem}
\begin{problem}
\label{prob:originalSOC}
Find $\pi_k(\cdot),\ k\in\Z_0 ^{N-1}$ that minimize \cref{eq:SOCobjective99fuel} while obeying the SDE \cref{eq:SDE} and satisfying the chance constraints \cref{eq:stoStateConst,eq:stoStateConstTube,eq:stoControlConst,eq:stoControlRateConst}, terminal constraint \cref{eq:SOCterminal}, and state-triggered chance constraints \cref{eq:stateTriggeredCCApproachCone} under control policies given by \cref{eq:SOCPolicy} with the state estimates \cref{eq:SOCfiltering}.
\end{problem}

\section{Solution Method via Convex Programming}

\subsection{Linear State Statistics Dynamics}
\label{sec:linearDynamics}

\subsubsection{Linear, discrete-time system}
Linearizing \cref{eq:SDE} about the reference state $\bm{x}^*(t) $ and control $\bm{u}^*(t) $, we have
\begin{align}
& \diff\bm{x} = 
(A \bm{x} + B^* \bm{u} + \bm{c} + B^*\widetilde{\bm{u}}^* )\diff t + G^* \diff\bm{w}(t),
\nonumber
\\
&A =
\left(
{\partial \bm{f}}/{\partial \bm{x}} 
\right)^*
,\ 
\bm{c} =
\left(
 \bm{f} - A \bm{x} - B \bm{u}
\right)^*
\label{eq:SDElin}
\end{align}
where 
$(\cdot)^*$ indicates the evaluation at $\bm{x}^*(t), \bm{u}^*(t)$.

Integrating \cref{eq:SDElin} over an interval $[t_k,t_{k+1})$ yields
\begin{align}\begin{aligned}
\hspace{-3pt}
\bm{x}_{k+1} = A_k\bm{x}_k + B_k \bm{u}_k + \bm{c}_k + G_{\mathrm{exe}, k} \bm{w}_{\mathrm{exe}, k}  + G_{k} \bm{w}_{k} 
\label{eq:SDElinDiscrete}
\end{aligned}\end{align}
where $\bm{w}_{\mathrm{exe}, k} \sim\mathcal{N}(0, I_{n_u} )$ and $\bm{w}_k \sim\mathcal{N}(0, I_{n_x} ) $.
The system matrices $A_k, B_k, \bm{c}_k, G_{\mathrm{exe}, k} $ are given by:
\begin{align}
&A_k = \Phi(t_{k+1}, t_k)
\nonumber
,\\
&
B_k = 
\begin{cases}
A_k B^*(t_k) & \text{(impulsive)} \\
A_k \int_{t_k}^{t_{k+1}}\Phi^{-1}(t, t_k) B^*(t) \diff t & \text{(continuous)}
\end{cases}
\label{eq:ABc_k}
,\\
&
\bm{c}_k = A_k \int_{t_k}^{t_{k+1}}\Phi^{-1}(t, t_k) \bm{c}(t) \diff t
,\  
G_{\mathrm{exe}, k} = B_k G_{\mathrm{exe}}(\bm{u}_k^*)
\nonumber
\end{align}
while $G_{k}$ is any matrix such that $G_{k} \bm{w}_k $ has covariance
\begin{align}\begin{aligned}
A_k \left\{
\int_{t_k}^{t_{k+1}} \Phi^{-1}(t, t_k) G(t) [\Phi^{-1}(t, t_k) G(t)]^\top \diff t
\right\} A_k^\top
\nonumber
\end{aligned}\end{align}
Note that $\Phi(t, t_k)$ denotes a state transition matrix from $t_k$ to $t$, obtained by solving the ordinary differential equation:
\begin{align}\begin{aligned}
\dot{\Phi}(t, t_k) = A(t)\Phi(t, t_k),
\quad
\Phi(t_k, t_k) = I
\label{eq:STM}
\end{aligned}\end{align}

If $\bm{u}_k^*=0$, then $T(\bm{u}_k^*) $ in \cref{eq:GatesModel} is undefined and so is $G_{\mathrm{exe}}(\bm{u}_k^*)$;
in such a case, we model it as $T(\bm{u}_k^*) = I_3 $.

\subsubsection{Filtered state dynamics}
We assume that the spacecraft is equipped with a Kalman filter to calculate navigation solutions onboard.
Hence, \cref{eq:SOCobservations} is approximated as:
\begin{align}
&\bm{y}_k = C_k \bm{x}_k + D_k\bm{w}_{\mathrm{obs}, k} + \bm{c}_{\mathrm{obs},k}
,\ \mathrm{with}\ 
C_k = [{\partial \bm{f}_{\mathrm{obs}}}/{\partial \bm{x}}]^*
,
\nonumber
\\ 
&D_k = G_{\mathrm{obs}} (\bm{x}_k^*)
,\ 
\bm{c}_{\mathrm{obs},k} = \bm{f}_{\mathrm{obs}}(\bm{x}_k^*) - C_k \bm{x}_k^*
\label{eq:linOBS}
\end{align}
Likewise, $\bm{f}_{\mathrm{obs}}(\hat{\bm{x}}_k^-) \approx \bm{c}_{\mathrm{obs},k} + C_k\hat{\bm{x}}_k^-$, leading \cref{eq:innovationProcess} to
\begin{align}\begin{aligned}
\widetilde{\bm{y}}_k^{-} = \bm{y}_k - (\bm{c}_{\mathrm{obs},k} + C_k\hat{\bm{x}}_k^-) 
= C_k\widetilde{\bm{x}}_k + D_k\bm{w}_{\mathrm{obs}, k}
\label{eq:lin_innovation}
\end{aligned}\end{align}
whose distribution is derived as (with $\widetilde{P}_{k}\triangleq\cov{\widetilde{\bm{x}}_k}$):
\begin{align}\begin{aligned}
\widetilde{\bm{y}}_k^{-} \sim \mathcal{N}(0, P_{\widetilde{y}_k^-}),\ 
P_{\widetilde{y}_k^-} = C_k \widetilde{P}_k^{-} C_k^\top + D_k D_k^\top
\label{eq:innovationCovariance}
\end{aligned}\end{align}
The Kalman filter sequentially updates the state estimate as:
\begin{align}
\hat{\bm{x}}_{k}^{-} &= A_{k-1}\hat{\bm{x}}_{k-1} + B_{k-1} \bm{u}_{k-1} + \bm{c}_{k-1} 
\quad  (\mathrm{time\ update})
\nonumber
\\
\hat{\bm{x}}_{k} &= \hat{\bm{x}}_{k}^{-} + L_k\widetilde{\bm{y}}_k^{-}
\quad  (\mathrm{measurement\ update})
\label{eq:KF}
\end{align}
which can be combined to yield
\begin{align}\begin{aligned}
\hat{\bm{x}}_{k+1} = A_k\hat{\bm{x}}_k + B_k \bm{u}_k + \bm{c}_k + L_{k+1} \widetilde{\bm{y}}_{k+1}^{-} 
\label{eq:linEstProcess}
\end{aligned}\end{align}
where $\hat{\bm{x}}_0 = \hat{\bm{x}}_0^{-} + L_0 \widetilde{\bm{y}}_0^-$.
Here, $L_k$ is the Kalman gain:
\begin{align}\begin{aligned}
L_k = \widetilde{P}_k^{-}C_k^\top(C_k \widetilde{P}_k^{-} C_k^\top + D_k D_k^\top)^{-1}
\end{aligned}\end{align}
which can be analytically calculated \textit{a priori} in linear Kalman filter
since $\widetilde{P}_{k}$ and $\widetilde{P}_{k}^{-}$ are also available \textit{a priori}:
\begin{align}
\hspace{-5pt}
\widetilde{P}_{k}^{-} &= A_{k-1}\widetilde{P}_{k-1}A_{k-1}^\top + G_{\mathrm{exe}, k-1}G_{\mathrm{exe}, k-1}^\top + G_{k-1}G_{k-1}^\top 
\nonumber
\\
\widetilde{P}_{k} 	&= (I - L_k C_k) \widetilde{P}_{k}^{-} (I - L_k C_k)^\top + L_kD_k D_k^\top L_k^\top
\label{eq:estimateCovariance}
\end{align}

\subsubsection{Linear output-feedback control policy}
We model the control policy \cref{eq:SOCPolicy} in a linear output feedback form as:
\begin{align}\begin{aligned}
\bm{u}_k &= \ols{\bm{u}}_k + K_{k}\bm{z}_k, \ k\in\Z_0^{N-1}
\label{eq:SOClinPolicy}
\end{aligned}\end{align}
where $\ols{\bm{u}}_k  $ is a nominal control input, $K_k$ is a feedback gain matrix, and $\bm{z}_k$ is a stochastic process given by
\begin{align}\begin{aligned}
\bm{z}_{k+1} = A_k\bm{z}_k + L_{k+1}\widetilde{\bm{y}}_{k+1}^-,
\quad
\bm{z}_0 = \hat{\bm{x}}_0 - \ols{\bm{x}}_0.
\label{eq:z-process}
\end{aligned}\end{align}
This approach, originally proposed by the author and collaborators in \cite{Aleti2023}, exploits the system's Markovian property.

\subsubsection{State and control statistics}
Let us first analytically derive the statistics---mean and covariance---of the state and control.
We express \cref{eq:linEstProcess} in a block-matrix form as:
\begin{align}\begin{aligned}
&
\begin{bmatrix}
\hat{\bm{x}}_0 \\
\hat{\bm{x}}_1 \\
\hat{\bm{x}}_2 \\
\vdots
\end{bmatrix}
=
\begin{bmatrix}
I_{n_x} \\
A_0 \\
A_1 A_0 \\
\vdots
\end{bmatrix}
\hat{\bm{x}}_0^{-}
+
\begin{bmatrix}
0   & 0 & \\
B_0                 & 0 & \\
A_1 B_0             & B_1               & \\
					&                   & {\ddots}
\end{bmatrix}
\begin{bmatrix}
\bm{u}_0 \\
\bm{u}_1 \\
\bm{u}_2 \\
\vdots
\end{bmatrix}
+
\\
&
\begin{bmatrix}
0   & 0 & \\
I_{n_x} 		 	& 0 & \\
A_1                 & I_{n_x} 			& \\
					&                   & {\ddots}
\end{bmatrix}
\begin{bmatrix}
\bm{c}_0 \\
\bm{c}_1 \\
\bm{c}_2 \\
\vdots
\end{bmatrix}
+
\begin{bmatrix}
L_0  				& 0 & \\
A_0L_0  			& L_1 				& \\
A_1A_0L_0 			& A_1L_1			& \\
					&                   & {\ddots}
\end{bmatrix}
\begin{bmatrix}
\widetilde{\bm{y}}_0^{-} \\
\widetilde{\bm{y}}_1^{-} \\
\widetilde{\bm{y}}_2^{-} \\
\vdots
\end{bmatrix}
\nonumber
\end{aligned}\end{align}
which can be expressed in a compact form as:
\begin{align}\begin{aligned}
\hat{\bm{X}} = 
\mathbf{A}\hat{\bm{x}}_0^{-} + 
\mathbf{B}\bm{U} + 
\mathbf{C} + 
\mathbf{L}{\bm{Y}},
\label{eq:blockMatrixEq}
\end{aligned}\end{align}
where 
$\hat{\bm{X}} = [\hat{\bm{x}}_0^\top, \hat{\bm{x}}_1^\top, ..., \hat{\bm{x}}_N^\top]^\top $,
$\bm{U} = [\bm{u}_0^\top, \bm{u}_1^\top, ..., \bm{u}_{N-1}^\top]^\top $,
$\bm{Y} = [\widetilde{\bm{y}}_0^{-1^\top}, \widetilde{\bm{y}}_1^{-1^\top}, ..., \widetilde{\bm{y}}_{N}^{-1^\top}]^\top $,
and $\mathbf{A},\mathbf{B},\mathbf{C}, \mathbf{L}$ are defined accordingly.
See \cite{Okamoto2018a,Ridderhof2020} on this construction.
Here, we define matrices $E_{x_k}$ ($E_{u_k} $) to extract $\bm{x}_k$ ($\bm{u}_k$) from $\bm{X} $ ($\bm{U} $) as:
\begin{align}\begin{aligned}
&\bm{x}_k = E_{x_k} \bm{X},\ 
\bm{u}_k = E_{u_k} \bm{U}.
\label{eq:extractMatrix}
\end{aligned}\end{align}

Under the control policy \cref{eq:SOClinPolicy}, $\bm{U}$ is given by:
\begin{align}\begin{aligned}
&\bm{U} = \ols{\bm{U}} + \mathbf{K}\bm{Z}
,\ 
\mathbf{K} =
\begin{bmatrix}
K_{0}  	& 0 	& 0 	& \hdots & 0\\
0  		& K_1 	& \ddots 		& \hdots & 0\\
\vdots 			& \vdots 		& \ddots 		& 0 & 0 \\
0  		& 0 	& \hdots 		& K_{N-1} & 0\\
\end{bmatrix}
\nonumber
\end{aligned}\end{align}
where 
$\bm{Z} = [\bm{z}_0^\top, \bm{z}_1^\top, ..., \bm{z}_{N}^\top]^\top $ is calculated as:
\begin{align}\begin{aligned}
&\bm{Z} = \mathbf{A}\bm{z}_0 + \mathbf{L}_Z\bm{Y}
\\
&\mathbf{L}_Z =
\begin{bmatrix}
0 & 0  				& 0 & 0 & \\
0 & L_1  			& 0 	& 0			& \\
0 & A_1L_1 			& L_2			& \\
0 & A_2A_1L_1 		& A_2L_2	& L_3		& \\
					&          & &         & {\ddots}
\end{bmatrix}
\label{eq:z-processBlock}
\end{aligned}\end{align}
Hence, \cref{eq:blockMatrixEq} under \cref{eq:SOClinPolicy} is expressed as:
\begin{align}\begin{aligned}
\hat{\bm{X}} = 
\mathbf{A}\hat{\bm{x}}_0^{-} + 
\mathbf{B}\ols{\bm{U}} + 
\mathbf{B}\mathbf{K}\bm{Z} + 
\mathbf{C} + 
\mathbf{L}{\bm{Y}},
\label{eq:blockMatrixEqWithOutputControl}
\end{aligned}\end{align}
which, since $\E[\bm{Y}] = 0$, $\E[\bm{Z}] = 0$, and $\E[\bm{U}] = \ols{\bm{U}}$, implies
\begin{align}
&
\E[\hat{\bm{X}}] \triangleq \ols{\bm{X}} =
 \mathbf{A}\ols{\bm{x}}_0 + \mathbf{B}\ols{\bm{U}} + \mathbf{C}.
\label{eq:blockStateMean}
\end{align}

From $\bm{z}_0 = \hat{\bm{x}}_0 - \ols{\bm{x}}_0$ and $\hat{\bm{x}}_0 = \hat{\bm{x}}_0^{-} + L_0 \widetilde{\bm{y}}_0^-$, we have
\begin{align}\begin{aligned}
\bm{Z} 
&= \mathbf{A}(\hat{\bm{x}}_0^{-} - \ols{\bm{x}}_0 + L_0 \widetilde{\bm{y}}_0^-) + \mathbf{L}_Z\bm{Y}\\
&= \mathbf{A}(\hat{\bm{x}}_0^{-} - \ols{\bm{x}}_0) + \mathbf{L}\bm{Y}
\label{eq:z-processBlock2}
\end{aligned}\end{align}
where $\mathbf{A}L_0 \widetilde{\bm{y}}_0^- + \mathbf{L}_Z\bm{Y} = \mathbf{L}\bm{Y} $.
From \cref{eq:z-processBlock2,eq:blockStateMean,eq:blockMatrixEqWithOutputControl},
\begin{align}\begin{aligned}
\hat{\bm{X}} - \ols{\bm{X}} 
&=
(I + \mathbf{B}\mathbf{K}) [\mathbf{A}(\hat{\bm{x}}_0^- - \ols{\bm{x}}_0) + \mathbf{L}{\bm{Y}}]
\label{eq:blockStateDeviation}
\end{aligned}\end{align}

Thus, in addition to the state mean \cref{eq:blockStateMean}, the state estimate covariance, $\hat{\mathbf{P}}$, is calculated as:
\begin{align}\begin{aligned}
\hat{\mathbf{P}}  \triangleq \cov{\hat{\bm{X}}}
=
(I + \mathbf{BK}) \mathbf{S} (I + \mathbf{BK})^\top
\label{eq:blockStateCov}
\end{aligned}\end{align}
where, noting the independency between $\hat{\bm{x}}_0, \widetilde{\bm{x}}_0, \bm{w}_{\mathrm{obs}, k}\forall k $,
\begin{align}\begin{aligned}
\hspace{-5pt}
\mathbf{S} 
\triangleq 
\cov{\mathbf{A}(\hat{\bm{x}}_0^{-} - \ols{\bm{x}}_0) + \mathbf{L}\bm{Y}}
= \mathbf{A}\hat{P}_{0^-}\mathbf{A}^\top + \mathbf{L}\mathbf{P}_Y\mathbf{L}^\top
\hspace{-5pt}
\label{eq:Sdefinition}
\end{aligned}\end{align}
Here, $\mathbf{P}_Y$, the innovation process covariance, is given by
\begin{align}\begin{aligned}
\mathbf{P}_Y \triangleq \cov{{\bm{Y}}} = \mathrm{blkdiag}(P_{\widetilde{y}_0^{-}}, P_{\widetilde{y}_1^{-}}, ..., P_{\widetilde{y}_N^{-}})
\end{aligned}\end{align}
where $\mathrm{blkdiag}(\cdot)$ forms a block diagonal matrix, and $P_{\widetilde{y}_k^{-}}$ is given in \cref{eq:innovationCovariance}.
Using \cref{eq:z-processBlock2}, $\mathbf{U} - \ols{\mathbf{U}} = \mathbf{K}\bm{Z} = \mathbf{K}[\mathbf{A}(\hat{\bm{x}}_0^{-} - \ols{\bm{x}}_0) + \mathbf{L}\bm{Y}] $, the control covariance is derived as:
\begin{align}\begin{aligned}
\mathbf{P}_U \triangleq \cov{{\bm{U}}} 
= \mathbf{KSK}^\top
\label{eq:controlCovariance}
\end{aligned}\end{align}

Now, we are ready to show a key result, given in \cref{theorem:affineStatistics}, to express the statistics of state and control in terms of the decision variables $\ols{\bm{U}} $ and $\mathbf{K}$ in an affine form.

\begin{proposition}
\label{theorem:affineStatistics}
Under the filtered state dynamics \cref{eq:linEstProcess} with the output-feedback control policy \cref{eq:SOClinPolicy}, $\ols{\bm{x}}_k$, $\hat{P}_k^{1/2} $, $P_k^{1/2} $, and ${P}_{u_k}^{1/2}$ are affine in $\ols{\bm{U}} $ and $\mathbf{K}$, given by:
\begin{align}
&\ols{\bm{x}}_k = E_{x_k} (\mathbf{A}\ols{\bm{x}}_0 + \mathbf{B}\ols{\bm{U}} + \mathbf{C}),\ 
\hat{P}_k^{1/2} = E_{x_k} (I + \mathbf{BK}) \mathbf{S}^{1/2},
\nonumber
\\ &
P_k^{1/2} =
\begin{bmatrix}
\hat{P}_k^{1/2} & \widetilde{P}_k^{1/2}
\end{bmatrix}
,\ 
{P}_{u_k}^{1/2} = E_{u_k} \mathbf{K} \mathbf{S}^{1/2}
\label{eq:sqrtCovariance}
\end{align}
where $\mathbf{S}^{1/2}$ is 
given by
$\mathbf{S}^{1/2} = 
\begin{bmatrix}
\mathbf{A}\hat{P}_{0^-}^{1/2} & \mathbf{L}\mathbf{P}_Y^{1/2}
\end{bmatrix}
$.
\end{proposition}
\begin{proof}
Combine \cref{eq:extractMatrix,eq:blockStateMean,eq:blockStateCov,eq:controlCovariance,eq:Sdefinition}, and note that
$P_k = \hat{P}_k + \widetilde{P}_k $ and $P_k = P_k^{1/2}[P_k^{1/2}]^\top $, where $\widetilde{P}_k$ is given by \cref{eq:estimateCovariance}.
\end{proof}

\subsection{Convex Problem Formulation}

\cref{theorem:jointChanceConstraint,theorem:hyperplaneUnion,theorem:3Dcc} are useful for convex formulation.

\begin{lemma}
\label{theorem:jointChanceConstraint}
Let $\bm{\xi}\sim \mathcal{N}(\ols{\bm{\xi}}, P_\xi) \in\R^{n_{\xi}} $.
Then a chance constraint 
$\P{\sum_{j} c_j(\bm{\xi}) \leq 0} \geq 1 - \varepsilon$
is implied by
\begin{align}\begin{aligned}
&\P{c_j(\bm{\xi}) \leq 0} \geq 1 - \varepsilon_j,\ \forall j,
\quad 
\sum_{j} \varepsilon_{j} \leq \varepsilon
\label{eq:jointCC}
\end{aligned}\end{align}
\end{lemma}
\begin{proof}
Use the Boole's inequality.
See \cite{Nemirovski2006}.
\end{proof}

\begin{lemma}
\label{theorem:hyperplaneUnion}
Let $\bm{\xi}\sim \mathcal{N}(\ols{\bm{\xi}}, P_\xi) \in\R^{n_{\xi}} $ and $\varepsilon\in(0,0.5) $.
Then,
$\P{\bm{a}_j^\top\bm{\xi} + b_j \leq 0,\ \forall j} \geq 1 - \varepsilon$
is implied by
\begin{align}\begin{aligned}
\bm{a}_j^\top\ols{\bm{\xi}} + b_j + m_{\mathcal{N}}\left(\varepsilon_j\right)\sqrt{\bm{a}^\top P_{\xi} \bm{a}}
\leq 0,\ \forall j,
\  
\sum_{j} \varepsilon_{j} \leq \varepsilon
\label{eq:ccHyperplaneConvex}
\end{aligned}\end{align}
$m_{\mathcal{N}}(\varepsilon) = Q_{X\sim\mathcal{N}}(1 - \varepsilon)$ denotes the quantile function of the standard normal distribution, evaluated at probability $(1 - \varepsilon)$.
\end{lemma}
\begin{proof}
Use the property of normal distribution.
See \cite{Ono2013}. 
\end{proof}

\begin{lemma}
\label{theorem:3Dcc}
Let $\bm{\xi}\sim \mathcal{N}(\ols{\bm{\xi}}, P_\xi) \in\R^{n_{\xi}} $, $\varepsilon\in(0,1) $, and $\gamma>0$.
Then,
$\P{\norm{\bm{\xi}}_2 \leq \gamma} \geq 1 - \varepsilon$
is implied by
\begin{align}\begin{aligned}
\norm{\ols{\bm{\xi}}}_2 + m_{\chi^2}\left(\varepsilon,n_\xi\right)\sqrt{\lambda_{\max}(P_{\xi})}
\leq \gamma
\label{eq:ccNormConstConvex}
\end{aligned}\end{align}
$m_{\chi^2}(\varepsilon,n_\xi) = \sqrt{Q_{X\sim \chi^2(n_\xi)}(1 - \varepsilon)}$.
Here, $Q_{X\sim \chi^2(n_\xi)}(1 - \varepsilon)$ denotes the quantile function of the chi-squared distribution with $n_\xi$ degrees of freedom, evaluated at probability $(1 - \varepsilon)$.
\end{lemma}
\if\shortOrFull1 
\begin{proof}[Proof]
Use the triangle inequality.
See \cite{Oguri2024a} for the proof.
\end{proof}
\fi
\if\shortOrFull2 
\begin{proof}[Proof (originally by the author in \cite{Oguri2022f})]
Denoting $\bm{\xi}$ as $\bm{\xi} = \bar{\bm{\xi}} + P_\xi^{1/2}\bm{v} $ where $\bm{v}\sim\mathcal{N}(0,I_{n_\xi}) $, and
applying the triangle inequality, we have $\normfs{\bm{\xi}}_2 \leq \normfs{\bar{\bm{\xi}}}_2 + \normfs{P_\xi^{1/2}\bm{v}}_2 \leq \normfs{\bar{\bm{\xi}}}_2 + \normfs{P_\xi^{1/2}}_2\normfs{\bm{v}}_2 $.
Thus,
$\P{\normfs{\bm{\xi}}_2\leq \gamma} \geq 
\P{\normfs{\bm{v}}_2 \leq (\gamma - \normfs{\bar{\bm{\xi}}}_2)/\normfs{P_\xi^{1/2}}_2 } \cdots $(A).
Noting $\normfs{\bm{v}}_2^2 \sim \chi^2(n_\xi)$, for a deterministic quantity $v_{\max} \geq 0 $, we have
$\P{\normfs{\bm{v}}_2 \leq v_{\max} } =
\P{\normfs{\bm{v}}_2^2 \leq v_{\max}^2 } 
\geq 1 - \varepsilon$, which is equivalent to
$m_{\chi^2}(\varepsilon,n_\xi) = \sqrt{Q_{X\sim \chi^2(n_\xi)}(1 - \varepsilon)} \leq v_{\max} \cdots$(B).
From (A), $\P{\normfs{\bm{\xi}}_2\leq \gamma} \geq 1 - \varepsilon $ is implied by
$\P{\normfs{\bm{v}}_2 \leq (\gamma - \normfs{\bar{\bm{\xi}}}_2)/\normfs{P_\xi^{1/2}}_2 } \geq 1 - \varepsilon$,
which, due to (B), is equivalent to
$m_{\chi^2}(\varepsilon,n_\xi) \leq {(\gamma - \normfs{\bar{\bm{\xi}}}_2)}/{\normfs{P_\xi^{1/2}}_2} \Leftrightarrow $ \cref{eq:ccNormConstConvex}, where note that $\sqrt{\lambda_{\max}(P_{\xi_k})} = \norm{P_{\xi_k}^{1/2}}_2 $.
\end{proof}
\fi

$m_{\mathcal{N}}$ and $m_{\chi^2}$ are straightforward to calculate in modern programming languages.
Use $\mathtt{norminv}(1 - \varepsilon)$ in Matlab and $\mathtt{stats.norm.ppf}(1 - \varepsilon)$ in Python's \texttt{scipy}
for $Q_{X\sim\mathcal{N}}(1 - \varepsilon)$;
$\mathtt{chi2inv}(1 - \varepsilon, n_\xi)$ in Matlab and $\mathtt{chi2.cdf}(1 - \varepsilon, n_\xi)$ in \texttt{scipy}
for $Q_{X\sim \chi^2(n_\xi)}(1 - \varepsilon)$.

\begin{remark}
$m_{\chi^2}$ given in \cref{theorem:3Dcc} provides a tighter approximation than the one proposed in \cite{Ridderhof2020c}.
The values of $m_{\chi^2}$ in \cref{theorem:3Dcc} and \cite{Ridderhof2020c} coincide when $n_\xi=2$, whereas, for $n_\xi>2$, $m_{\chi^2}$ in \cref{theorem:3Dcc} is smaller than \cite{Ridderhof2020c}, hence giving tighter constraints.
\cite{Oguri2022f} reports a quantitative comparison of the two approaches for various $n_\xi$ and $\varepsilon$.
\end{remark}

\subsubsection{Cost function}
\cref{eq:SOCobjective99fuel} in discrete-time is given by
\begin{align}
J = 
\begin{cases}
\sum_{k=0}^{N-1}Q_{X\sim\norm{\bm{u}_k}_2}(p) & (\mathrm{impulsive})
\\
\sum_{k=0}^{N-1}Q_{X\sim\norm{\bm{u}_k}_2}(p) \Delta t_k & (\mathrm{continuous})
\end{cases} 
\label{eq:SOCobjective99fuelDiscrete}
\end{align}
where
$\Delta t_k = t_{k+1} - t_{k}$.
While \cref{eq:SOCobjective99fuelDiscrete} is not easy to calculate, 
\cref{theorem:deltaV99} gives an upper bound of \cref{eq:SOCobjective99fuelDiscrete}.

\begin{lemma}
\label{theorem:deltaV99}
Suppose $\bm{u}_k \sim\mathcal{N}(\ols{\bm{u}}_k, P_{u_k}) $.
Then, $Q_{X\sim\norm{\bm{u}_k}_2}(p)$ in \cref{eq:SOCobjective99fuelDiscrete} is bounded from above as:
\begin{align}\begin{aligned}
Q_{X\sim\norm{\bm{u}_k}_2}(p) \leq \norm{\ols{\bm{u}}_k}_2 + m_{\chi^2}\left(1 - p,n_u\right) \sqrt{\lambda_{\max}(P_{u_k})}
\nonumber
\end{aligned}\end{align}
\end{lemma}

\if\shortOrFull1 
\begin{proof}[Proof]
Leverage \cref{theorem:3Dcc}.
See \cite{Oguri2024a} for the full proof.
\end{proof}
\fi

\if\shortOrFull2 
\begin{proof}
Using
\cref{theorem:3Dcc}, $\P{\norm{\bm{u}_k}_2 \leq a} \geq p$ is implied by
$\norm{\ols{\bm{u}}_k}_2 + m_{\chi^2}\left(1 - p,n_u\right) \sqrt{\lambda_{\max}(P_{u_k})} \leq a $.
Hence, there exists a non-negative scalar $\delta \geq 0 $ that satisfies
\begin{align}\begin{aligned}
&\P{\norm{\bm{u}_k}_2 \leq a} \geq p
\quad\Leftrightarrow \\
&\norm{\ols{\bm{u}}_k}_2 + m_{\chi^2}\left(1 - p,n_u\right) \sqrt{\lambda_{\max}(P_{u_k})} - \delta \leq a 
\end{aligned}\end{align}
which implies
$
\min\{a\in\R \mid \P{\norm{\bm{u}_k}_2 \leq a} \geq p\}
=
\norm{\ols{\bm{u}}_k}_2 + m_{\chi^2}\left(1 - p,n_u\right) \sqrt{\lambda_{\max}(P_{u_k})} - \delta
$, completing the proof. 
\end{proof}
\fi

Applying \cref{theorem:deltaV99} to \cref{eq:SOCobjective99fuelDiscrete},
we have $J \leq J_{\mathrm{ub}}$, where
\begin{align}\begin{aligned}
\hspace{-6pt}
J_{\mathrm{ub}} = 
\begin{cases}
\sum_{k}
[ \norm{\ols{\bm{u}}_k}_2 + m_{\chi^2}\left(0.01,n_u\right) \norm{P_{u_k}^{1/2} }_2 ] 
\\
\sum_{k}
[ \norm{\ols{\bm{u}}_k}_2 + m_{\chi^2}\left(0.01,n_u\right) \norm{P_{u_k}^{1/2} }_2 ]\Delta t_k 
\end{cases}
\label{eq:convexSOCcost}
\end{aligned}\end{align}
Instead of $J$, we minimize $J_{\mathrm{ub}}$ in \cref{eq:convexSOCcost} for tractability.

\subsubsection{Path chance constraints}
Using \cref{theorem:hyperplaneUnion}, it is clear that \cref{eq:stoStateConst} is implied by:
\begin{align}
\hspace{-5pt}
\bm{a}_j^\top \ols{\bm{x}}_k + b_j + m_{\mathcal{N}}(\varepsilon_{x,j}) \norm{\bm{a}^\top P_k^{1/2} }_2
\leq 0,\ \forall j,k(\in\Z_{0}^{N})
\label{eq:stoStateConstDetFormConvex}
\end{align}
with $\sum_{j} \varepsilon_{x,j} \leq \varepsilon_{x}$.
Using \cref{theorem:3Dcc}, it is clear that \cref{eq:stoControlConst,eq:stoStateConstTube} are respectively implied by:
\begin{align}
\hspace{-5pt}
&\norm{H(\ols{\bm{x}}_k - \bm{x}_k^*)}_2 + 
m_{\chi^2}\left(\varepsilon_x,n_h\right)\norm{H P_k^{1/2} }_2
\leq d_{\max}
\label{eq:stoStateConstTubeConvex}
\\
&\norm{\ols{\bm{u}}_k}_2 + m_{\chi^2}\left(\varepsilon_u,n_u\right)\norm{P_{u_k}^{1/2} }_2
\leq u_{\max}
,\ k\in\Z_{0}^{N-1}
\label{eq:stoControlConstDetFormConvex}
\end{align}

Noting that $\Delta \bm{u}_k = \bm{u}_{k+1} - \bm{u}_k = (E_{u_{k+1}} - E_{u_k}) \bm{U} $ and applying \cref{theorem:3Dcc}, \cref{eq:stoControlRateConst} can be expressed as:
\begin{align}\begin{aligned}
&
m_{\chi^2}\left(\varepsilon_u,n_u\right)\norm{(E_{u_{k+1}} - E_{u_k})\mathbf{K}\mathbf{S}^{1/2}}
+
\\
&
\norm{(E_{u_{k+1}} - E_{u_k}) \ols{\bm{U}}}_2
\leq \Delta u_{\max}
,\ k\in\Z_{0}^{N-2}
\label{eq:stoControlRateConstDetFormConvex}
\end{aligned}\end{align}

\subsubsection{Terminal constraints}
Noting \cref{eq:sqrtCovariance}, the terminal mean constraint is equivalent to:
\begin{align}\begin{aligned}
E_{x_N}(\mathbf{A}\ols{\bm{x}}_0 + \mathbf{B}\ols{\bm{U}} + \bm{C}) - \ols{\bm{x}}_f = 0
\label{eq:convexTerminal1}
\end{aligned}\end{align}
Since $P_N = \hat{P}_N + \widetilde{P}_N$, the terminal covariance constraint leads to
$(0 \prec) \hat{P}_N \preceq P_f - \widetilde{P}_N $, which, by noting \cref{eq:sqrtCovariance}, can be equivalently expressed in a convex form as \cite{Okamoto2018a}:
\begin{align}\begin{aligned}
\norm{(P_f - \widetilde{P}_N)^{-1/2}E_{x_N} (I + \mathbf{BK})\mathbf{S}^{1/2} }_2 - 1 \leq 0.
\label{eq:convexTerminal2}
\end{aligned}\end{align}
where $P_f$ is defined by the user and $\widetilde{P}_N$ is given by \cref{eq:estimateCovariance}.

\subsubsection{State-triggered chance constraints}
\cref{eq:stateTriggeredCCApproachCone} is expressed in deterministic form as:
\begin{align}
\text{if }
g_{\mathrm{stc}}(\ols{\bm{x}}_k)<0
,\text{ then }
c_{\mathrm{stc}}(\ols{\bm{x}}_k, P_k^{1/2}) \leq 0
,\ \forall k\in\Z_{0}^{N}
\label{eq:trigerCCconvex}
\end{align}
where 
$g_{\mathrm{stc}}(\cdot) = \norm{H_r\ols{\bm{x}}_k}_2 - r_c$,
and, using \cref{theorem:hyperplaneUnion,theorem:3Dcc,theorem:jointChanceConstraint},
\begin{align}
&c_{\mathrm{stc}}(\cdot) = 
\norm{A_{\mathrm{cone}} H_r \ols{\bm{x}}_k}_2 - \bm{b}_{\mathrm{cone}}^\top H_r \ols{\bm{x}}_k
\label{eq:c_stc}
\\
&+ m_{\chi_2}(\frac{\epsilon_x}{2}, 2)\norm{A_{\mathrm{cone}} H_r P_k^{1/2}}_2
+ m_{\mathcal{N}}(\frac{\epsilon_x}{2})\norm{\bm{b}_{\mathrm{cone}}^\top H_r P_k^{1/2}}_2
\nonumber
\end{align}

Using \cref{eq:stateTriggered}, \cref{eq:trigerCCconvex} is logically equivalent to
\begin{align}\begin{aligned}
\hspace{-5pt}
h_{\mathrm{stc}}(\cdot) =
- \min\{g_{\mathrm{stc}}(\ols{\bm{x}}_k), 0 \}\cdot c_{\mathrm{stc}}(\ols{\bm{x}}_k, P_k^{1/2}) \leq 0
\label{eq:stc_non-convex}
\end{aligned}\end{align}
$\forall k\in\Z_{0}^{N}$, being non-convex because of the multiplication of convex functions.
Hence, we approximate \cref{eq:stc_non-convex} as:
\begin{align}\begin{aligned}
\hspace{-5pt}
h_{\mathrm{stc}}(\cdot)
=
- \min\{g_{\mathrm{stc}}(\ols{\bm{x}}_k^*), 0 \} \cdot
 c_{\mathrm{stc}}(\ols{\bm{x}}_k, P_k^{1/2}) \leq 0
,\forall k
\label{eq:stc_convex}
\end{aligned}\end{align}

\subsection{Convex Chance-Constrained Path Planning Problem}
If the problem does not involve a state-triggered constraint, 
we find a history of chance-constrained control policies by solving \cref{prob:convexSOC}, which is convex as in \cref{theorem:convexSOC}.
\begin{problem}
\label{prob:convexSOC}
Find $\ols{\bm{U}} $ and $\mathbf{K} $ that minimize the cost \cref{eq:convexSOCcost} such that satisfy the path chance constraints \cref{eq:stoControlConstDetFormConvex,eq:stoStateConstDetFormConvex,eq:stoControlRateConstDetFormConvex,eq:stoStateConstTubeConvex}, terminal constraints \cref{eq:convexTerminal1,eq:convexTerminal2}.
\end{problem}

\begin{theorem}
\label{theorem:convexSOC}
\cref{prob:convexSOC} is a convex optimization problem.
\end{theorem}
\begin{proof}
Combine \cref{theorem:affineStatistics} and \cref{eq:convexSOCcost,eq:stoControlConstDetFormConvex,eq:stoStateConstDetFormConvex,eq:stoControlRateConstDetFormConvex,eq:stoStateConstTubeConvex,eq:convexTerminal1,eq:convexTerminal2}.
\end{proof}

If the problem involves state-triggered constraints, we approximate \cref{eq:stc_non-convex} via \cref{eq:stc_convex} for convex formulation.
To avoid \textit{artificial infeasibility} due to the approximation, we introduce slack variables $\zeta_k \in\R$ and relax \cref{eq:stc_convex} as:
\begin{align}\begin{aligned}
h_{\mathrm{stc}}(\cdot)
=
- \min\{g_{\mathrm{stc}}(\ols{\bm{x}}_k^*), 0 \} \cdot
 c_{\mathrm{stc}}(\ols{\bm{x}}_k, P_k^{1/2}) \leq \zeta_k
\label{eq:stc_convex_relax}
\end{aligned}\end{align}
$\forall k\in\Z_{0}^{N}$, while penalizing the constraint violation in the cost function by introducing a penalty weight $w\in\R$ as:
\begin{align}\begin{aligned}
\hspace{-6pt}
J_{\mathrm{penalty}}(\ols{\bm{U}}, \mathbf{K}, \bm{\zeta}) = 
J_{\mathrm{ub}}(\ols{\bm{U}}, \mathbf{K})
+
w \norm{\bm{\zeta}}_1
\label{eq:convexSOCcostPenalty}
\end{aligned}\end{align}
where $\bm{\zeta} = [\zeta_0, ..., \zeta_N]^\top $.
Control policies with state-triggered constraints are found by iteratively solving \cref{prob:convexSOC_stc}, which is convex due to \cref{theorem:convexSOC_stc};
at every iteration, $g_{\mathrm{stc}}(\ols{\bm{x}}_k^*),\forall k\in\Z_{0}^{N} $ in \cref{eq:stc_convex_relax} are updated by using the previous solution.
The iteration is terminated when the updates in $\ols{\bm{X}} $ and $\ols{\bm{U}} $ become smaller than a tolerance or when the number of iterations reaches a pre-determined number.

It is also possible to take a more sophisticated sequential convex programming (SCP) approach, e.g., \texttt{SCvx} \cite{Mao2016} and \texttt{SCvx*} \cite{Oguri2023b}, which helps ensure the convergence.

\begin{problem}
\label{prob:convexSOC_stc}
Find $\ols{\bm{U}} $, $\mathbf{K} $, $\bm{\zeta}$ that minimize the cost \cref{eq:convexSOCcostPenalty} such that satisfy the path chance constraints \cref{eq:stoControlConstDetFormConvex,eq:stoStateConstDetFormConvex,eq:stoControlRateConstDetFormConvex,eq:stoStateConstTubeConvex}, terminal constraints \cref{eq:convexTerminal1,eq:convexTerminal2}, and state-triggered chance constraint \cref{eq:stc_convex_relax}.
\end{problem}

\begin{theorem}
\label{theorem:convexSOC_stc}
\cref{prob:convexSOC_stc} is a convex optimization problem.
\end{theorem}
\begin{proof}
Combine \cref{theorem:affineStatistics} and \cref{eq:convexSOCcost,eq:convexSOCcostPenalty,eq:stoControlConstDetFormConvex,eq:stoStateConstDetFormConvex,eq:stoControlRateConstDetFormConvex,eq:stoStateConstTubeConvex,eq:convexTerminal1,eq:convexTerminal2,eq:stc_convex_relax,eq:c_stc}; note $\min\{g_{\mathrm{stc}}(\ols{\bm{x}}_k^*), 0 \} \leq 0 $ in \cref{eq:stc_convex_relax}.
\end{proof}


\section{Numerical Examples}
\label{sec:example}

\subsection{Safe Autonomous Rendezvous under Uncertainty}

\if\shortOrFull1 
\begin{figure}[tb] 
    \centering
    \vspace{-5pt}
    \subfigure[\label{f:CWH-1} Position 2-D projection: MC result with predicted statistics]{
        \includegraphics[width=0.95\linewidth]{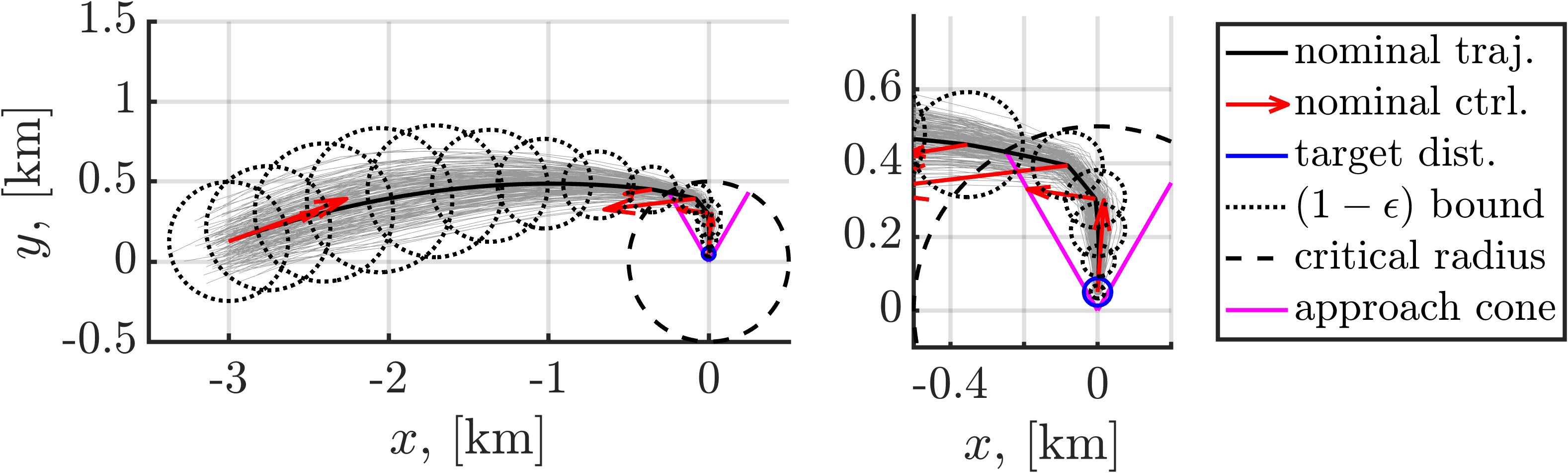}}
    \vspace{-5pt}
    \subfigure[\label{f:CWH-2} MC results with predicted statistics, in (m/s)]{
        \includegraphics[width=0.75\linewidth]{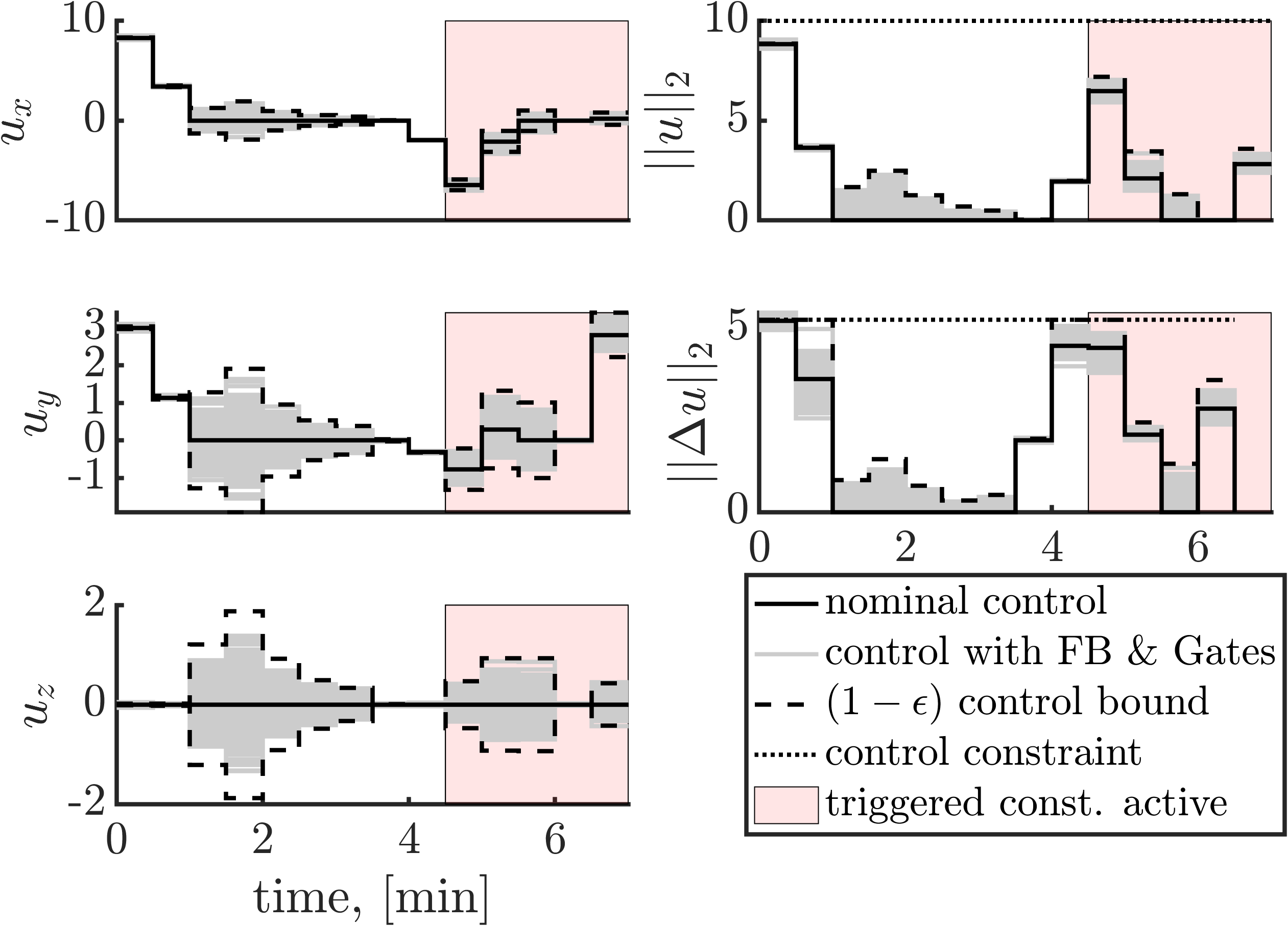}}
    \subfigure[\label{f:CWH-3} MC total $\Delta V$ histogram \& statistics]{
        \includegraphics[width=0.65\linewidth]{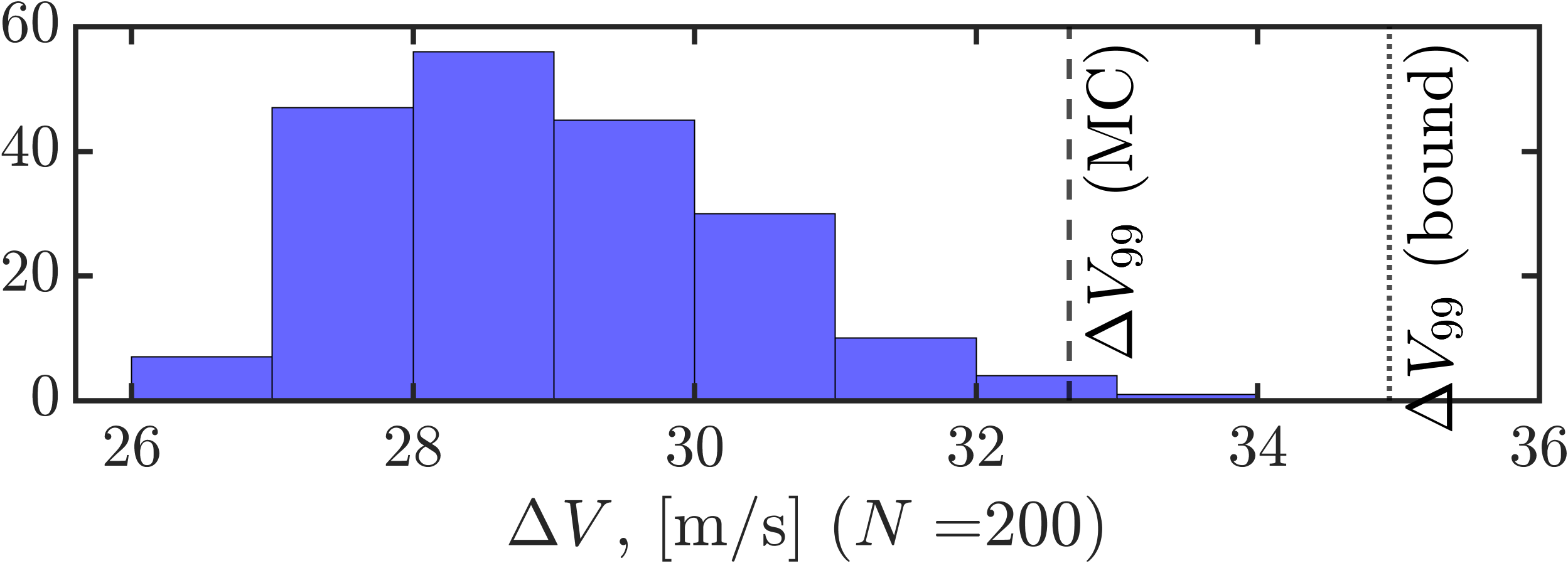}}
    \vspace{-5pt}
    \caption{\label{f:CWH} Safe autonomous rendezvous with state-trigger chance constraints.
    }
    \vspace{-15pt}
\end{figure}
\fi

\if\shortOrFull1 
Consider a safe autonomous rendezvous scenario in proximity of a chief satellite in an Earth orbit under operational uncertainties, with impulsive maneuvers.
As the spacecraft is close to the chief, it is appropriate to use CWH equation to approximate the EoMs.
The constraints considered are \cref{eq:stoControlConst,eq:stoControlRateConst,eq:SOCterminal,eq:stateTriggeredCCApproachCone};
the state-triggered chance constraint models an approach cone constraint that is activated when the spacecraft is closer to the chief than a threshold.
See \cite{Oguri2024a} for the simulation setting detail.
\fi
\if\shortOrFull2 
Consider a safe autonomous rendezvous scenario in proximity of a chief satellite in an Earth orbit under operational uncertainties, with impulsive maneuvers.
As the spacecraft is close to the chief, it is appropriate to use CWH equation to approximate the EoMs.
This scenario features 
chance constraints on 
control magnitude \cref{eq:stoControlConst},
control rate \cref{eq:stoControlRateConst},
terminal distribution \cref{eq:SOCterminal},
and
state-triggered chance constraint \cref{eq:stateTriggeredCCApproachCone}, which models an approach cone constraint that is activated when the spacecraft is closer to the chief than a threshold.
The used parameters are included in Appendix.
\fi

\if\shortOrFull2 
\begin{figure}[tb] 
    \centering
    \subfigure[\label{f:CWH-1} Position 2-D projection: MC result with predicted statistics]{
        \includegraphics[width=\linewidth]{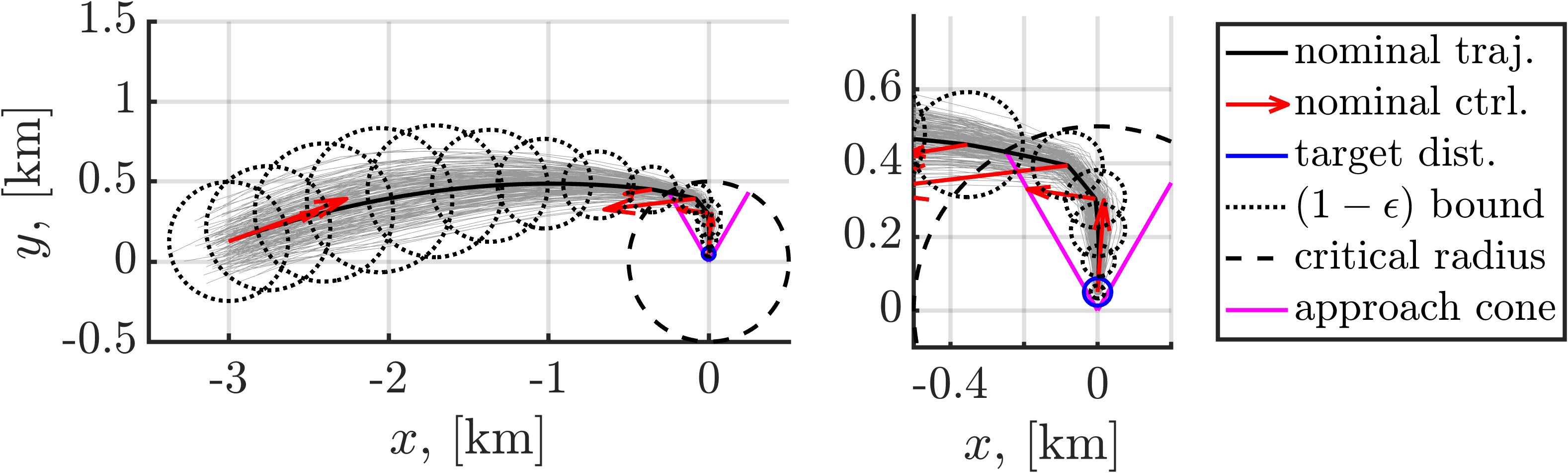}}
    \subfigure[\label{f:CWH-2} MC results with predicted statistics, in (m/s)]{
        \includegraphics[width=0.85\linewidth]{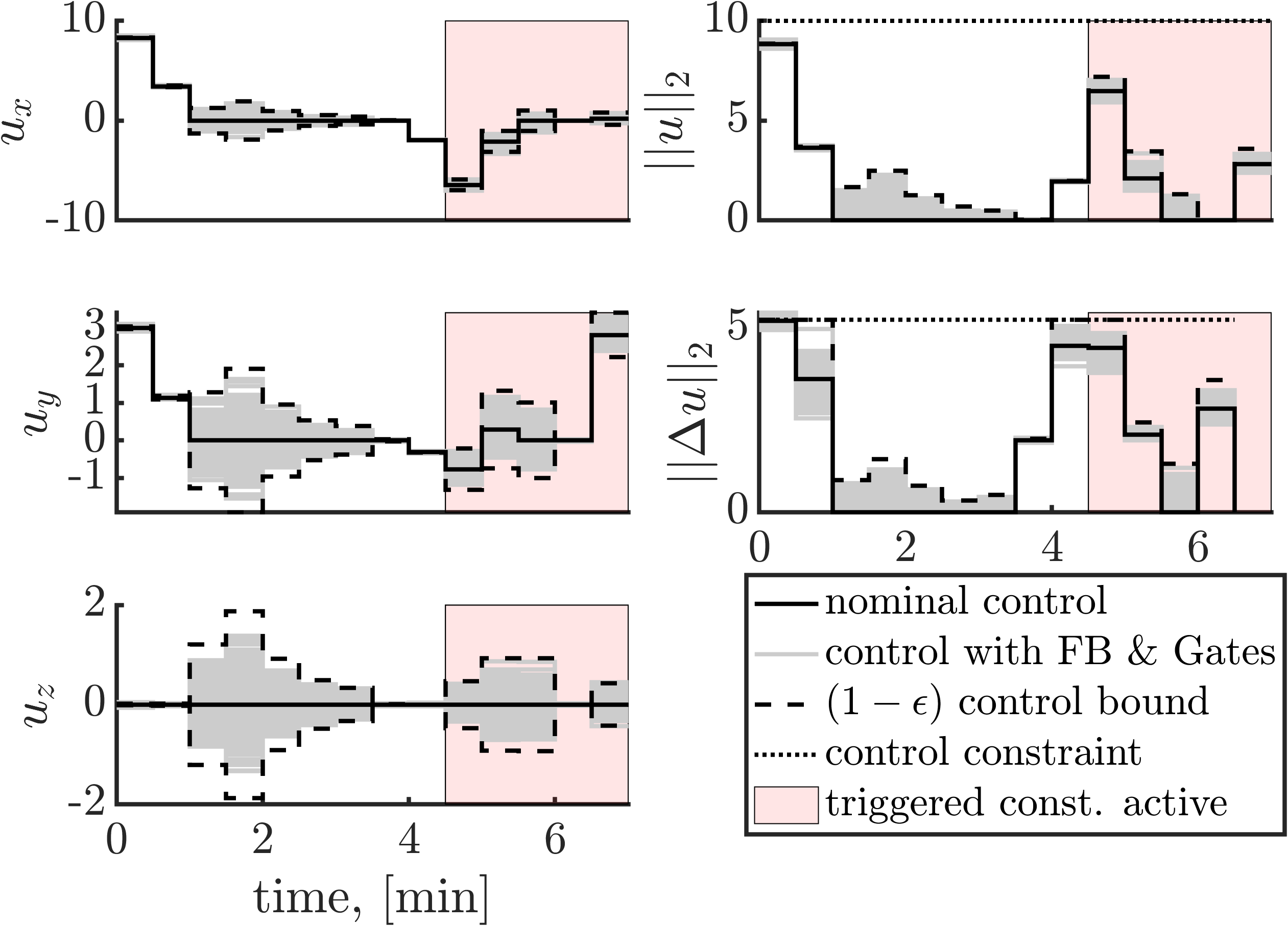}}
    \subfigure[\label{f:CWH-3} MC total $\Delta V$ histogram \& statistics]{
        \includegraphics[width=0.7\linewidth]{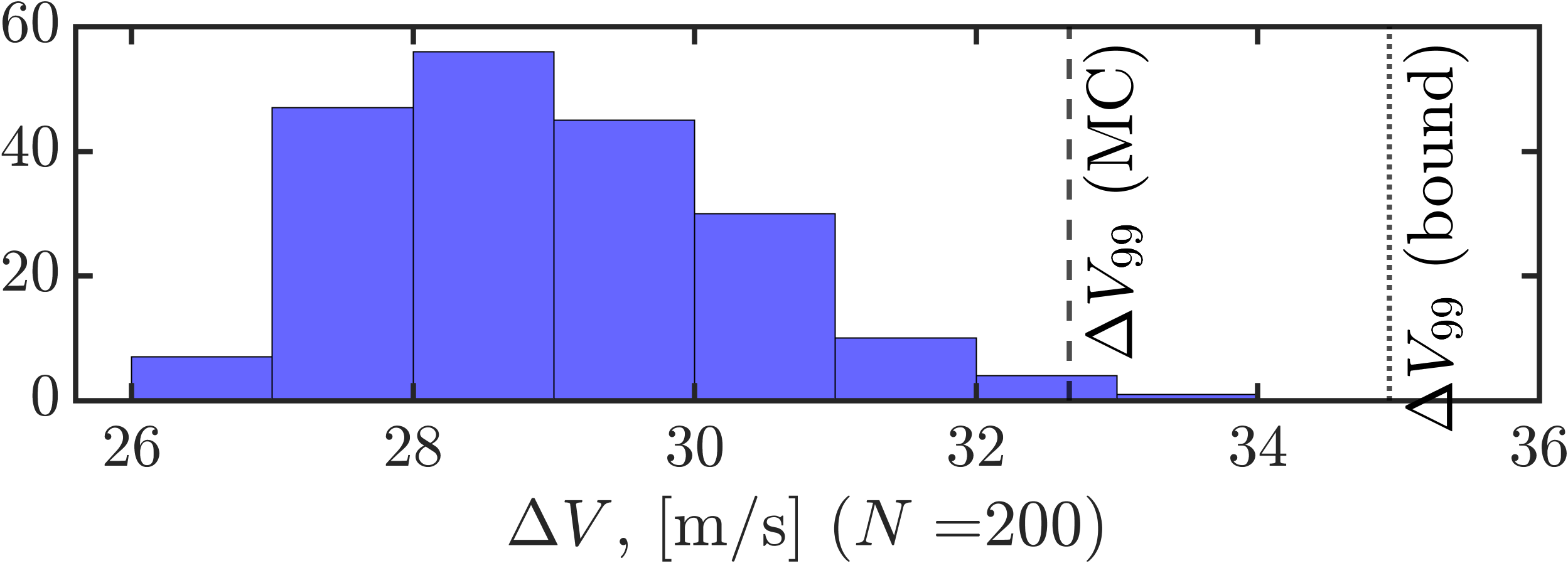}}
    \caption{\label{f:CWH} Safe autonomous rendezvous scenario with the approach cone constraint triggered when entering a sphere of critical radius $r_{\mathrm{trigger}}$.
    }
    \vspace{-15pt}
\end{figure}
\fi

\if\shortOrFull1 
\begin{figure}[tb] 
    \centering
    \subfigure[\label{f:NRHO-1} Nonlinear MC trajectories \& statistics ($y$-$z$ projection)]{
        \includegraphics[width=0.8\linewidth]{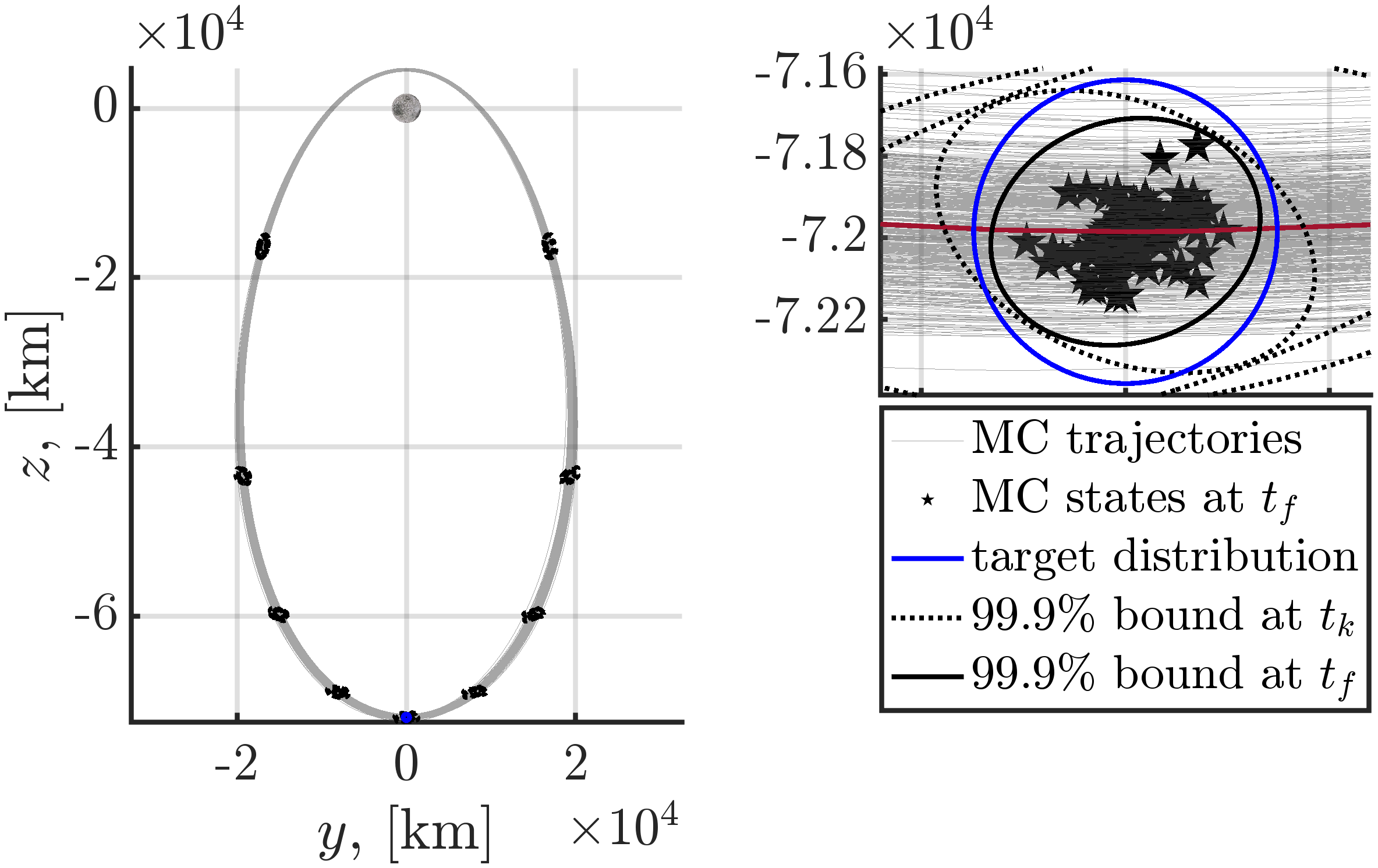}}
    \vspace{-5pt}
    \subfigure[\label{f:NRHO-2} MC position deviation;  discrete-time c.c. successfully met]{
        \includegraphics[width=0.8\linewidth]{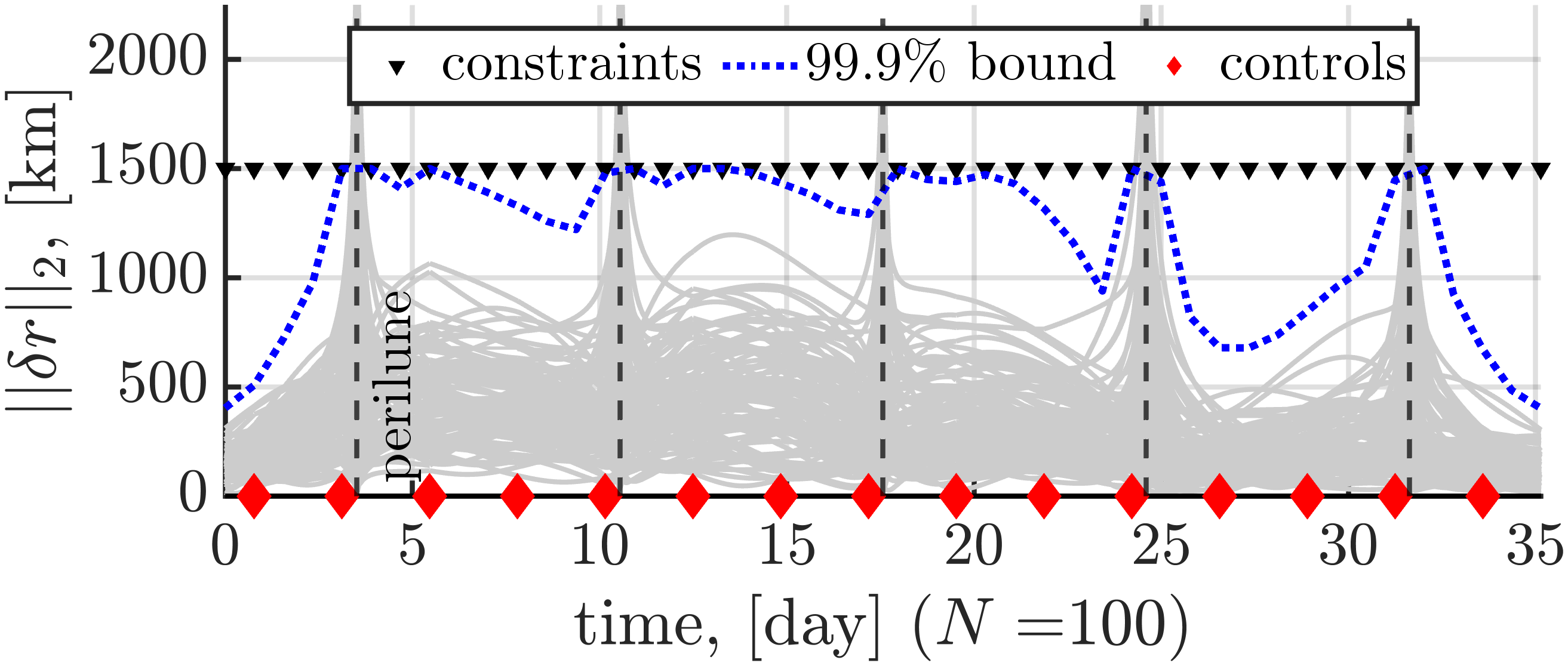}}
    \subfigure[\label{f:NRHO-3} MC total $\Delta V$ histogram \& statistics]{
        \includegraphics[width=0.65\linewidth]{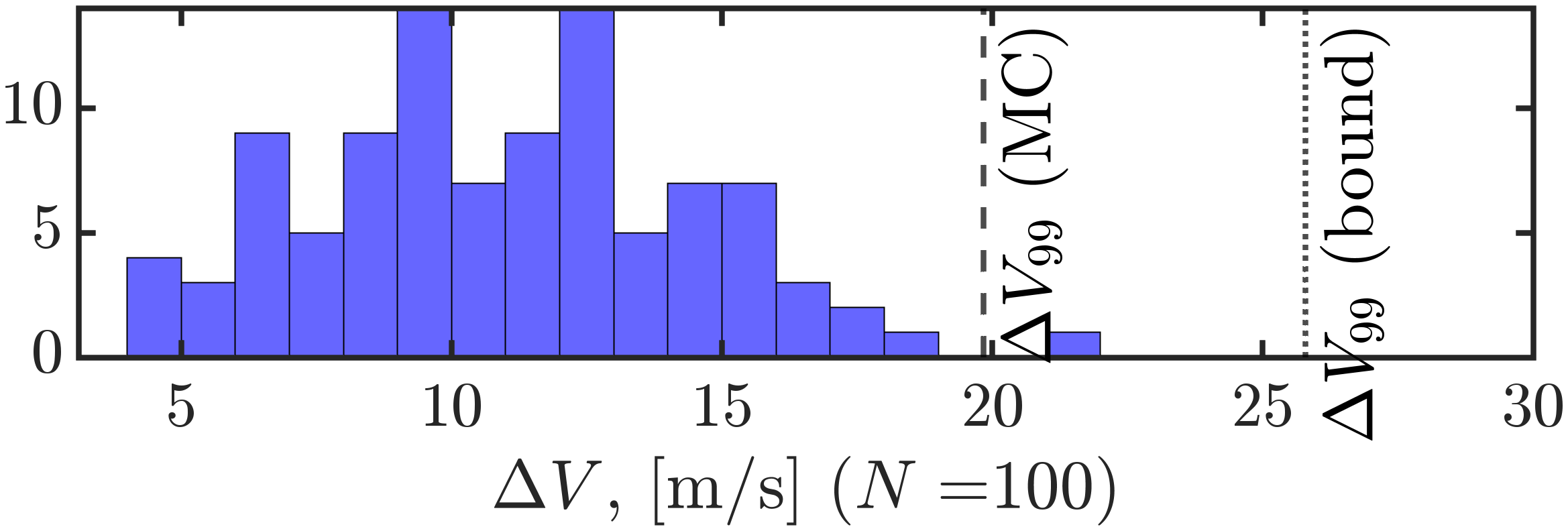}}
    \vspace{-5pt}
    \caption{\label{f:NRHO} Safe autonomous cislunar L2 NRHO station-keeping scenario
    }
    \vspace{-15pt}
\end{figure}
\fi

As the problem involves a state-triggered constraint, \cref{prob:convexSOC_stc} is solved iteratively until the variable update is smaller than a tolerance $\epsilon_{\mathrm{tol}} {=} 10^{-3}$.
Each convex programming took $2.8{-}3.6$ seconds and the iteration was terminated by satisfying $\epsilon_{\mathrm{tol}}$ after 5 iterations.
Monte Carlo (MC) simulation is performed using the designed control policies $\pi_{0:N-1} $.

\cref{f:CWH} summarizes the MC result.
\cref{f:CWH-1} illustrates trajectories projected on the 2-D position space while \cref{f:CWH-2} highlights $\bm{u}_k $, $\norm{\bm{u}_k}_2 $, and $\norm{\Delta \bm{u}_k}_2 $.
These figures demonstrate that the designed policies $\pi_{1:N} $ successfully drive the state to the target distribution while satisfying all the chance constraints.
\cref{f:CWH-3} shows that $\Delta V_{99}$ from MC is indeed upper-bounded by $J_{\mathrm{ub}} $ with some optimality gap (${\sim}2 $ m/s).

\if\shortOrFull2 
\begin{figure}[tb] 
    \centering
    \subfigure[\label{f:NRHO-1} Nonlinear MC trajectories \& predicted statistics ($y$-$z$ projection)]{
        \includegraphics[width=0.9\linewidth]{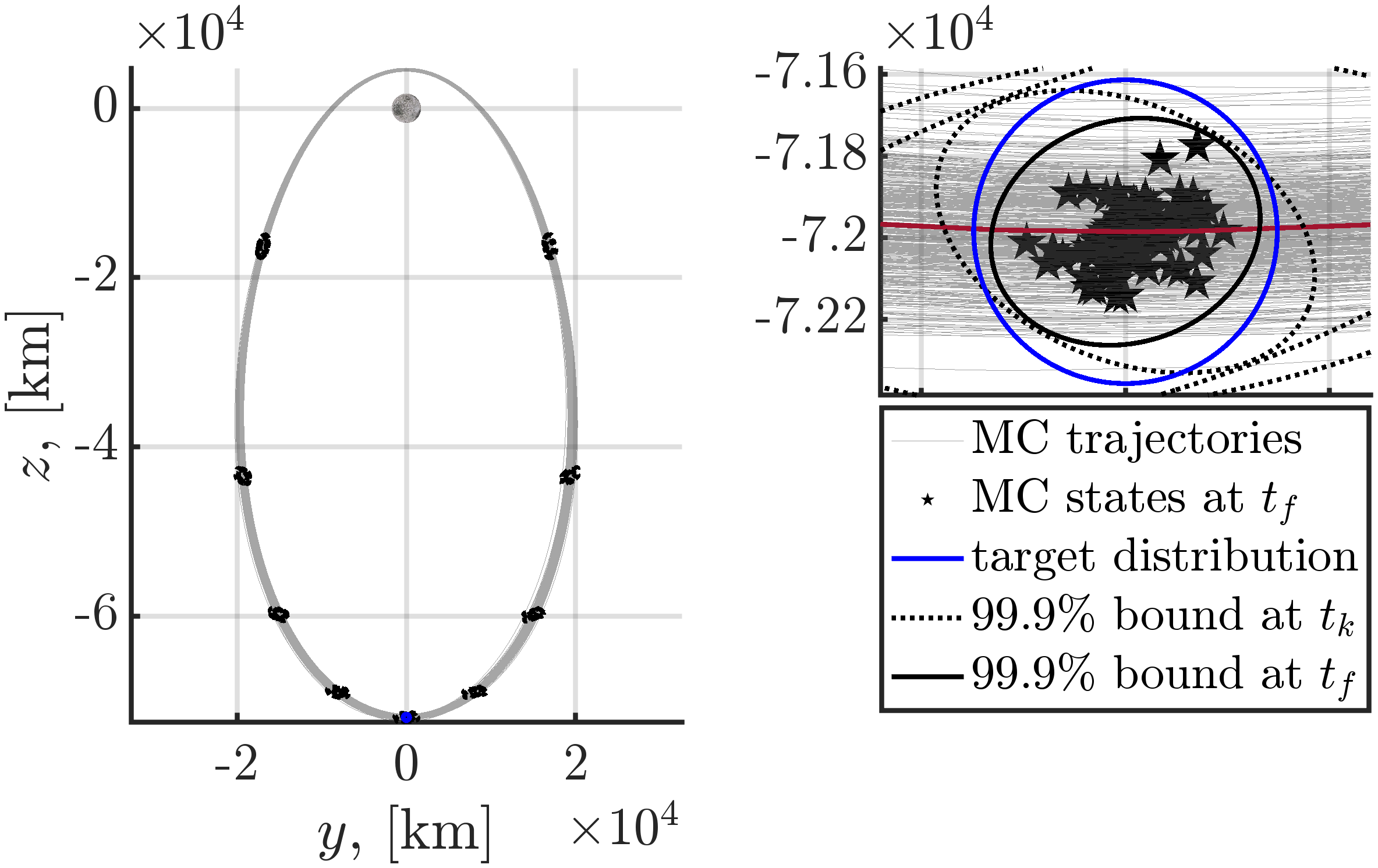}}
    \subfigure[\label{f:NRHO-2} MC position deviation;  discrete-time c.c. successfully met]{
        \includegraphics[width=0.9\linewidth]{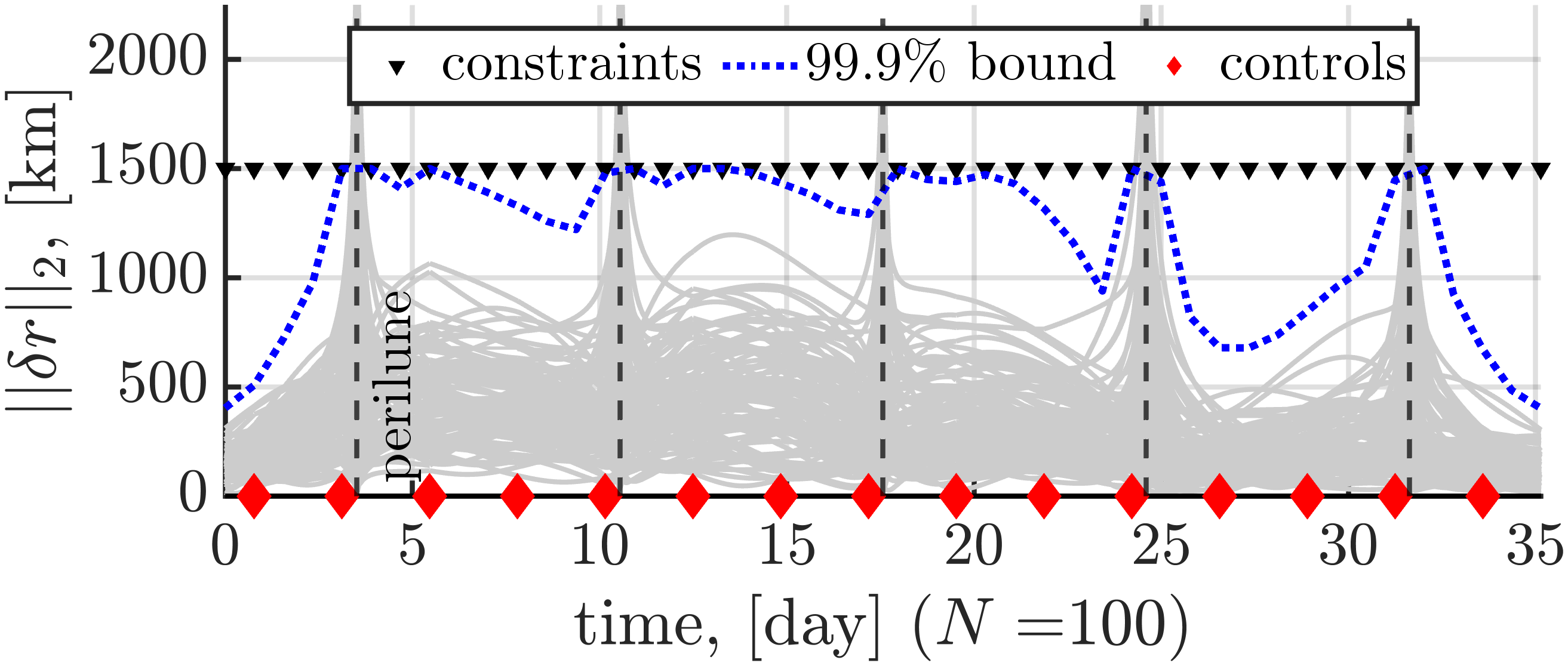}}
    \subfigure[\label{f:NRHO-3} MC total $\Delta V$ histogram \& statistics]{
        \includegraphics[width=0.7\linewidth]{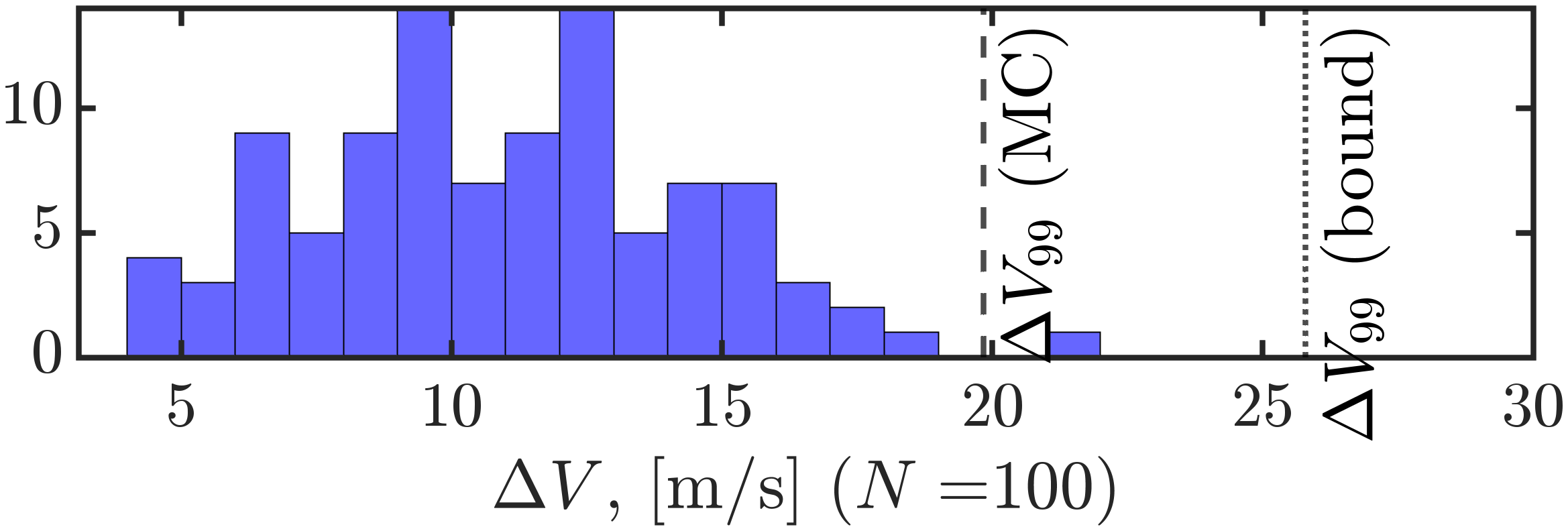}}
    \caption{\label{f:NRHO} Safe autonomous cislunar L2 NRHO station-keeping scenario
    }
    \vspace{-15pt}
\end{figure}
\fi

\subsection{Safe NRHO Station-keeping Planning in Cislunar Space}
\label{sec:NRHOexample}

Next consider safe station-keeping planning on a cislunar NRHO under uncertainty.
For measurement modeling \cref{eq:SOCobservations}, we consider Moon horizon-based optical navigation (OpNav), which is originally proposed in \cite{Christian2016} and applied to various cislunar orbits in \cite{Qi2023a}; we use the same OpNav parameters and filtering architecture as described in \cite{Qi2023a}.

\if\shortOrFull1 
To apply the proposed framework, we linearize the nonlinear CR3BP EoM about a reference trajectory.
We impose \cref{eq:stoStateConstTube} to constrain position deviations from the reference NRHO and \cref{eq:stoControlConst} to constrain $\norm{\bm{u}_k}_2 $.
The scenario begins at the apolune and considers 5 NRHO revolutions, where 9 nodes are placed per orbit, evenly spaced in time, yielding $\Delta t \approx 0.8$ day.
The spacecraft takes a measurement (moon image) at every node, and has maneuver opportunities every 3 nodes.
See \cite{Oguri2024a} for the simulation setting detail.
\fi
\if\shortOrFull2 
To apply the proposed framework, we linearize the nonlinear CR3BP EoM about a reference trajectory.
We impose a tube-like constraint in \cref{eq:stoStateConstTube} to constrain position deviations from the reference NRHO and control magnitude constraint in \cref{eq:stoControlConst}.
The scenario begins at the apolune and considers 5 full NRHO revolutions, where $N=45$, i.e., 9 nodes are placed per orbit, evenly spaced in time, leading to $\Delta t \approx 0.8$ day.
The spacecraft takes a measurement (moon image) at every node, and has trajectory correction maneuver (TCM) opportunities at every 3 nodes, that is, 3 TCM opportunities per orbit.
Tube-like state chance constraints are imposed at every node.
See Appendix for the parameters for NRHO, uncertainty modeling, and constraints.
\fi

Solving this problem does not require SCP as it does not have a state-triggered constraint, and is solved in ${\sim}8.4$ sec, producing about 35-day-worth safe control policy $\pi_k$.

\cref{f:NRHO} summarizes nonlinear MC results, verifying the robustness of $\pi_k$.
The MC simulation nonlinearly evaluates $\bm{x}_k $, $\hat{\bm{x}}_k $, $\widetilde{\bm{y}}_k^{-} $, and $\bm{z}_k$:
the nonlinear CR3BP EoM for $\bm{x}_k $ and OpNav-based extended Kalman filter (EKF) for $\hat{\bm{x}}_k $, $\widetilde{\bm{y}}_k^{-} $, and $\bm{z}_k$;
EKF evaluates \cref{eq:ABc_k,eq:STM,eq:linOBS} along $\hat{\bm{x}}_k$ and $\hat{\bm{u}}_k $, and uses \cref{eq:innovationProcess} instead of \cref{eq:lin_innovation}, which then affects $\bm{z}_k$ through \cref{eq:z-process}.
\cref{f:NRHO-1} shows the MC trajectories projected on $y${-}$z$ plane and highlights the distributional constraint being met at $t_f$.
\cref{f:NRHO-2} shows the satisfaction of the state chance constraints under the optimized policy.%
\footnote{Since chance constraints are imposed in discrete time, constraint violation may occur momentarily in-between constrained epochs. Constraint violations in \cref{f:NRHO-2} correspond to dynamically sensitive perilune passages.
See \cite{Oguri2019b} for continuous-time chance constraints.}
\cref{f:NRHO-3} indicates $\Delta V_{99}$ from MC is upper-bounded by $J_{\mathrm{ub}} $.

\section{Conclusions}
A safe spacecraft path-planning problem under uncertainty is formulated as an output-feedback, chance-constrained optimal control problem.
The presented formulation exploits the Markovian property of the system and incorporates various chance constraints, including state-triggered chance constraints.
The proposed planner designs a history of safe control policies by solving (sequential) convex programming, minimizing 99\% quantile of control cost while ensuring the vehicle safety under uncertainties.
The formulation is validated via safe autonomous rendezvous in proximity operations and safe cislunar NRHO station-keeping planning.

\section*{Acknowledgments}
The implementation of NRHO OpNav in \cref{sec:NRHOexample} is based on the initial development by Daniel Qi in \cite{Qi2023a}.


\if\shortOrFull1 
\fi
\if\shortOrFull2 
\section*{Appendix}
\label{sec:appendix}

\begin{table}[tb]
\centering
\caption{CWH rendezvous scenario uncertainty parameters}
\label{t:paramsCWH1}
\begin{tabular}{lcccccccccccc}
\toprule
Quantity 	& Symbol & Value & Unit \\ \hline
measurement error (pos.) 	& $\sigma_{\mathrm{obs},r}$ 	& 1.0 & m \\
measurement error (vel.) 	& $\sigma_{\mathrm{obs},v}$ 	& 0.01 & m/s \\
initial dispersion (pos.) 	& $\sigma_r$ 					& 100& m \\
initial dispersion (vel.) 	& $\sigma_v$ 					& 1.0& m/s \\
stochastic acceleration  		& $\sigma_a$ 					& 1.0& $\mathrm{mm/s^{3/2}}$ \\
execution error (fixed mag.) & $\sigma_1$ 					& 1.0& cm/s \\
execution error (prop. mag.) & $\sigma_2$ 					& 1.0& \% \\
execution error (fixed point.) & $\sigma_3$ 					& 1.0& cm/s \\
execution error (prop. point.) & $\sigma_4$ 					& 1.0& deg
\\ \bottomrule
\end{tabular}
\end{table}

\begin{table}[tb]
\centering
\caption{CWH rendezvous scenario constraint parameters}
\label{t:paramsCWH2}
\begin{tabular}{lcccccccccccc}
\toprule
Quantity 	& Symbol & Value & Unit \\ \hline
mean pos. at $t_0$ & $\ols{\bm{r}}_0$					& $[-3.0,\, 0.126,\ 0]^\top$ & km \\
mean vel. at $t_0$ & $\ols{\bm{v}}_0$					& $\bm{0}_{3\times1}$ & km/s \\
mean pos. at $t_f$ & $\ols{\bm{r}}_f$					& $[0.0,\, 0.05,\ 0]^\top$ & km \\
mean vel. at $t_f$ & $\ols{\bm{v}}_f$					& $\bm{0}_{3\times1}$ & km/s \\
pos. std. dev. at $t_f$ & $\sigma_{r_f}$ 				& 10.0& m \\
vel. std. dev. at $t_f$ & $\sigma_{v_f}$ 				& 0.1& m/s \\
max $\Delta V$ magnitude & $u_{\max} $ 					& 10.0& m/s \\
max attitude rate & $\omega_{\max} $ 					& 1.0& deg/s \\
approach cone angle & $\theta_{\max} $ 					& 30.0& deg \\
trigger critical radius & $r_{\mathrm{trigger}} $ & 0.5& km \\
rick bound (state) & $\epsilon_x$ & $10^{-3}$ 			& - \\
rick bound (control) & $\epsilon_u$ & $10^{-3}$ 			& - 
\\ \bottomrule
\end{tabular}
\end{table}

\subsection{Safe constrained rendezvous parameters}
We consider $r_0 = 7228~\mathrm{km}$ for the chief's orbit radius (${\sim}850$ km altitude).
The time is discretized with interval $\Delta t = 30~\mathrm{sec} $ with $N=14$.

We model the observation process as 
$\bm{f}_{\mathrm{obs}}(\bm{x}_k) = \bm{x}_k $ and $G_{\mathrm{obs}} = \mathrm{blkdiag}(\sigma_{\mathrm{obs},r}I_3, \sigma_{\mathrm{obs},v}I_3 ) $,
although the formulation can model more realistic measurements too (e.g., range, angle bearing).
See \cref{t:paramsCWH1,t:paramsCWH2} for specific values.

Parameters that define other uncertainties are as follows.
Initial state dispersion:
$\hat{P}_0^{-} =  \mathrm{blkdiag}(\sigma_{r}^2 I_3, \sigma_{v}^2 I_3 ) $.
Initial estimate error:
$\widetilde{P}_{0^-} =  \mathrm{blkdiag}(\sigma_{\mathrm{obs},r}^2 I_3, \sigma_{\mathrm{obs},v}^2 I_3 ) $.
Stochastic acceleration:
$G = [0_{3\times3}, \sigma_{a} I_3]^\top $.
The initial mean state and the target distribution are defined as:
$\ols{\bm{x}}_0 = [\ols{\bm{r}}_0^\top, \ols{\bm{v}}_0^\top]^\top $,
$\ols{\bm{x}}_f = [\ols{\bm{r}}_f^\top, \ols{\bm{v}}_f^\top]^\top $, 
and
$P_f =  \mathrm{blkdiag}(\sigma_{r_f}^2 I_3, \sigma_{v_f}^2 I_3 ) $.

Parameters that define constraints are as follows.
Maximum control rate: 
$\Delta u_{\max} = u_{\max} \omega_{\max} \Delta t $.
Approach cone parameters:
$A_{\mathrm{cone}} = 
\begin{bsmallmatrix}
1 & 0 & 0 \\
0 & 0 & 1
\end{bsmallmatrix} $
and 
$\bm{b}_{\mathrm{cone}} =
[0,  \tan{\theta_{\max}},  0]
^\top$,
allowing the spacecraft to approach the chief from $+y$ direction with the cone angle $2\theta_{\max} $.

\subsection{Safe NRHO station-keeping planning parameters}

The observation process is modeled via Moon horizon-based oOpNav and uses the same parameters and filtering architecture as described in \cite{Qi2023a}.
The parameters used for uncertainty modeling and constraints are listed in \cref{t:paramsNRHO1,t:paramsNRHO2}.
The NRHO initial condition is included in \cref{t:paramsNRHO2}, where the state vector is in the non-dimensional unit, with the characteristic length and time in the Earth-Moon CR3BP being roughly
$l_* = 3.84748 \times 10^{5} $~km and $t_* = 3.75700 \times 10^{5} $~sec.
Note that the initial conditions must be obtained through a differential correction process \cite{Pavlak2013} to design a tight multi-revolution NRHO for the use as a reference trajectory.

\begin{table}[tb]
\centering
\caption{NRHO station-keeping scenario uncertainty parameters}
\label{t:paramsNRHO1}
\begin{tabular}{lcccccccccccc}
\toprule
Quantity 	& Symbol & Value & Unit \\ \hline
initial dispersion (pos.) 	& $\sigma_r$ 					& 100& km \\
initial dispersion (vel.) 	& $\sigma_v$ 					& 1.0& m/s \\
stochastic acceleration  		& $\sigma_a$ 				& $10^{-4}$ & $\mathrm{mm/s^{3/2}}$ \\
execution error (fixed mag.) & $\sigma_1$ 					& 1.0 & cm/s \\
execution error (prop. mag.) & $\sigma_2$ 					& 1.0& \% \\
execution error (fixed point.) & $\sigma_3$ 					& 1.0 & cm/s \\
execution error (prop. point.) & $\sigma_4$ 					& 1.0& deg
\\ \bottomrule
\end{tabular}
\end{table}

\begin{table}[tb]
\centering
\caption{NRHO station-keeping scenario constraint parameters}
\label{t:paramsNRHO2}
\begin{tabular}{lcccccccccccc}
\toprule
Quantity 	& Symbol & Value & Unit \\ \hline
mean pos. at $t_0$ & $\ols{\bm{r}}_0$					& $[1.0300,\, 0.0,\ -0.1871]^\top$ & nd \\
mean vel. at $t_0$ & $\ols{\bm{v}}_0$					& $[0.0,\, -0.1200,\ 0.0]^\top$ & nd \\
mean pos. at $t_f$ & $\ols{\bm{r}}_f$					& $[1.0300,\, 0.0,\ -0.1871]^\top$ & nd \\
mean vel. at $t_f$ & $\ols{\bm{v}}_f$					& $[0.0,\, -0.1200,\ 0.0]^\top$ & nd \\
pos. std. dev. at $t_f$ & $\sigma_{r_f}$ 				& 100.0& km \\
vel. std. dev. at $t_f$ & $\sigma_{v_f}$ 				& 1.0& m/s \\
max $\Delta V$ magnitude & $u_{\max} $ 					& 5.0& m/s \\
max state deviation & $d_{\max} $ 						& 1500& km \\
rick bound (state) & $\epsilon_x$ & $10^{-3}$ 			& - \\
rick bound (control) & $\epsilon_u$ & $10^{-3}$ 			& - 
\\ \bottomrule
\end{tabular}
\end{table}

\fi

\bibliographystyle{ieeetr}

\if\shortOrFull1 
  \input{bbl_short}
\fi
\if\shortOrFull2 

\fi

\end{document}